\documentclass[12pt]{amsart}

\usepackage{tikz,amsthm,amsmath,amstext,amssymb,amscd,epsfig,euscript, mathrsfs, dsfont,pspicture,multicol,graphpap,graphics,graphicx,times,enumerate,subfig,sidecap,wrapfig,color,pict2e}

\usepackage[colorlinks,citecolor=cyan,pagebackref,hypertexnames=false]{hyperref}

\usepackage{enumerate}

 \numberwithin{equation}{section}

\def\bB{{\mathbb{B}}}

\def\bR{{\mathbb{R}}}
\def\R{{\mathbb{R}}}

\def\bZ{{\mathbb{Z}}}

\def\bN{{\mathbb{N}}}

\def\cA{{\mathscr{A}}}
\def\cB{{\mathscr{B}}}
\def\cC{{\mathscr{C}}}
\def\cD{{\mathscr{D}}}

\def\cG{{\mathscr{G}}}
\def\cH{{\mathscr{H}}}
\def\cI{{\mathscr{I}}}

\def\cP{{\mathscr{P}}}

\def\cV{{\mathscr{V}}}

\def\one{\mathds{1}}

\def\ve{\varepsilon}

\renewcommand{\d}{{\partial}}
\def\lec{\lesssim}
\def\gec{\gtrsim}

\DeclareMathOperator{\diam}{diam}
\def\Mod{\mathop\mathrm{Mod}} 					
\def\dim{\mathop\mathrm{dim}} 					
\def\dist{\mathop\mathrm{dist}} 						
\def\supp{\mathop\mathrm{supp}}					

\newcommand{\ps}[1]{\left( #1 \right)}

\newcommand{\ck}[1]{\left\{#1 \right\}}
\newcommand{\av}[1]{\left| #1 \right|}

\newcommand{\isif}[1]{\left\{\begin{array}{cc} #1
\end{array}\right.}

\def\warrow{\rightharpoonup}

\def\Xint#1{\mathchoice
{\XXint\displaystyle\textstyle{#1}}%
{\XXint\textstyle\scriptstyle{#1}}%
{\XXint\scriptstyle\scriptscriptstyle{#1}}%
{\XXint\scriptscriptstyle\scriptscriptstyle{#1}}%
\!\int}
\def\XXint#1#2#3{{\setbox0=\hbox{$#1{#2#3}{\int}$ }
\vcenter{\hbox{$#2#3$ }}\kern-.58\wd0}}

\def\avint{\Xint-}
\def\grad{\nabla}

\theoremstyle{plain}
\newtheorem*{maintheorem}{Main Theorem}
\newtheorem{theorem}{Theorem}

\newtheorem{corollary}[theorem]{Corollary}

\newtheorem{lemma}[theorem]{Lemma}

\theoremstyle{definition}

\numberwithin{equation}{section}
\numberwithin{theorem}{section}

  \DeclareFontFamily{U}{mathb}{\hyphenchar\font45} 
\DeclareFontShape{U}{mathb}{m}{n}{
      <5> <6> <7> <8> <9> <10> gen * mathb
      <10.95> mathb10 <12> <14.4> <17.28> <20.74> <24.88> mathb12
      }{}
\DeclareSymbolFont{mathb}{U}{mathb}{m}{n}

\DeclareMathSymbol{\toitself}{3}{mathb}{"FD}  

\makeatletter
\def\@tocline#1#2#3#4#5#6#7{\relax
  \ifnum #1>\c@tocdepth 
  \else
    \par \addpenalty\@secpenalty\addvspace{#2}%
    \begingroup \hyphenpenalty\@M
    \@ifempty{#4}{%
      \@tempdima\csname r@tocindent\number#1\endcsname\relax
    }{%
      \@tempdima#4\relax
    }%
    \parindent\z@ \leftskip#3\relax \advance\leftskip\@tempdima\relax
    \rightskip\@pnumwidth plus4em \parfillskip-\@pnumwidth
    #5\leavevmode\hskip-\@tempdima
      \ifcase #1
       \or\or \hskip 1em \or \hskip 2em \else \hskip 3em \fi%
      #6\nobreak\relax
    \dotfill\hbox to\@pnumwidth{\@tocpagenum{#7}}\par
    \nobreak
    \endgroup
  \fi}
\makeatother

\begin{document}

\title{Poincar\'e Inequalities and Uniform Rectifiability}

\author[Azzam]{Jonas Azzam}
\address{Jonas Azzam\\
School of Mathematics \\ University of Edinburgh \\ JCMB, Kings Buildings \\
Mayfield Road, Edinburgh,
EH9 3JZ, Scotland.}
\email{j.azzam "at" ed.ac.uk}


\begin{abstract}
We show that any $d$-Ahlfors regular subset of $\R^{n}$ supporting a weak $(1,d)$-Poincar\'{e} inequality with respect to surface measure is uniformly rectifiable.
\end{abstract}
\maketitle
\tableofcontents

\section{Introduction}

For $p\geq 1$, a metric measure space $X$ equal to the support of a  doubling measure $\mu$ admits a  {\it weak $(1,p)$-Poincar\'{e} inequality} if, for all measurable functions $u$ with constants $C,\lambda \geq 1$, we have 
 \[
 \avint_{B} |u-u_{B}|d\mu \leq C\diam B \ps{\avint_{\lambda B} \rho^{p}d\mu}^{\frac{1}{p}}
 \]
 where $\rho$ is any {\it upper gradient} for $u$, meaning for every $x,y\in X$,
 \[
 |u(x)-u(y)|\leq \int_{\gamma} \rho
 \]
 for any rectifiable curve $\gamma$ connecting $x$ to $y$ in $X$. This condition, introduced by Heinonen and Koskela in \cite{HK98},  is shared by a large class of metric spaces, such as the Heisenberg group and Ahlfors regular Riemannian manifolds of non-negative Ricci curvature,  and imposes certain geometric properties on the metric space $X$. For example, if $X$ is complete, then $X$ must be quasiconvex and in fact there are quantitatively many curves running through $X$ (we will be more specific about this later). 
 
By a result of Cheeger \cite{Che99}, such spaces also admit a differentiable structure that allows for a generalization of Rademacher's theorem.  Specifically, he showed there exist a countably family of Borel sets $\{U_{i}\}$ covering $\mu$-almost all of $X$ and Lipschitz maps $\phi_i:U_{i}\rightarrow \R^{d(i)}$ for some $d(i)\in \bN$ so that if $f:X\rightarrow \R$ is any Lipschitz function, then for $\mu$ almost every $x\in U_i$, there is a vector $df(x)\in \R^{d(i)}$ so that 

\[
\limsup_{y\rightarrow x} \frac{|f(y)-f(x) - df(x)\cdot (\phi_i(y)-\phi_i(x))|}{|y-x|}=0.
\]
(Here and below we will write $|x-y|$ rather than $d_{X}(x,y)$).  See also \cite{Kei04} for an improvement, \cite{KM16-primer} for a compact primer to this result, and \cite{BKO19} for a shorter proof using Guth's multilinear Kakeya inequality for neighbourhoods of Lipschitz graphs. Any complete metric measure space $(X,\mu)$ having such a collection of charts $(U_i,\phi_i)$ is called a {\it Lipschitz differentiability space}. 

There are many examples of spaces that have Poincar\'{e} inequalities (and hence a differentiable structure) but pathological geometric structure, such as the Heisenberg group and Laakso spaces \cite{Laa00}. Under certain conditions, however, these conditions imply nice structure. In particular, recall that a metric measure space $(X,\mu)$ is {\it $d$-rectifiable} if $X$ may be covered up to $\mu$-measure zero by countably many Lipschitz images of subsets of $\R^{d}$; we will say $X$ is $d$-rectifiable if $\cH^{d}|_X$ is $d$-rectifiable. Also recall that a metric measure space $(X,\mu)$ is {\it Ahlfors $d$-regular} if $X=\supp \mu$ where $\mu$ is Ahlfors regular, meaning there is $A >0$ so that for all $x\in X$,
\[
A^{-1}r^{d}\leq \mu(B(x,r))\leq Ar^{d} \;\; \mbox{ for all }\;\; x\in X, \;\; 0<r<\diam X.
\]

\begin{theorem}
\label{t:rectifiable}
If $(X,\mu)$ is an Ahlfors $d$-regular Lipschitz differentiability space so that $X\subseteq \R^{n}$, then $(X,\mu)$ is $d$-rectifiable. 
\end{theorem}

This is alluded to in \cite[p. 259-260]{Kei03}, we sketch the proof (although using later results). If $(U_i,\phi_i)$ is one of the chart maps in the differentiable structure for $(X,\mu)$ and $\phi:U\rightarrow \R^{d(i)}$, then by \cite[Theorem 14.1]{Che99}, since $(X,\mu)$ is Ahlfors $d$-regular, $d(i)\geq d$. By Cheeger's conjecture \cite[Theorem 1.1]{DPMR17}, $\phi_i[\mu|_{U_i}]\ll \mathscr{L}^{d(i)}$ where $\mathscr{L}^{d(i)}$ is Lebesgue measure on $\R^{d(i)}$ and $\phi_i[\mu]$ is the push-forward measure. Since $\dim X=d$ and $\phi_i$ is Lipschitz, $\dim \phi_i(U_i)\leq d$, so if $d(i)>d$, then $\mathscr{L}^{d(i)}(\phi_i(U_i))=0$ and Cheeger's conjecture implies that 
\[
\mu(U_i)=\phi_i[\mu|_{U_{i}}](\phi_i(U_i))=0.\] 
Thus, almost every $x\in X$ is contained in the domain of a chart map $(U_i,\phi_i)$ where $d(i)=d$, i.e. $\phi_i:U_i\rightarrow \R^{d}$. By another result of Cheeger \cite[Theorem 14.2]{Che99}, $\mu$ is $d$-rectifiable\footnote{It was pointed out to the author by the referee that in the statement of \cite[Theorem 14.2]{Che99}, the statement assumes $\cH^{k(\alpha)}(V_{\alpha})>0$, when really he needs $\cH^{k(\alpha)}(V_{\alpha})<\infty$.}.

The main result of this paper is that, if we assume a $d$-Poincar\'{e} inequality, we in fact obtain better rectifiable properties of $(X,\mu)$ when $X$ is Euclidean.

\begin{maintheorem}
Let $n>d\geq 2$ be integers and $(X,\mu)$ be a closed Ahlfors $d$-regular space in $\R^{n}$ with constant $A\geq 1$ supporting a weak $(1,d$)-Poincar\'{e} inequality with constants $C,\lambda \geq 1$. Then $X$ is {\it uniformly $d$-rectifiable} (UR), meaning there are constants $L,c>0$ so that for every $x\in X$ and $0<r<\diam X$, there is an $L$-bi-Lipschitz image of a subset of $\R^{d}$ contained in $X\cap B(x,r)$ of $\cH^{d}$-measure at least $cr^{d}$. The constants $L$ and $c$ only depend on $n,C,\lambda$, and $A$.
\end{maintheorem}

Uniformly rectifiable sets were introduced by David and Semmes in \cite{DS}, and are a sort of quantitative version of a rectifiable set, in the sense that $X$ is UR if it is rectifiable {\it by the same amount and Lipschitz constant} in every ball. They feature in various results that characterize when a certain quantitative property holds on an Ahlfors regular set. For example, certain classes of singular integral operators are bounded on an Ahlfors regular set if and only if that set is UR \cite{DS}. The result above (and results about UR sets) are {\it quantitative} in the sense that there {\it exists} a formula indicating the dependencies of constants, but typically they are not ``effective" in the sense that they produce explicit formulas or inequalities relating the constants. For many applications, however, this (ineffective) quantitative dependence is sufficient.

One previous result similar to our Main Theorem is due to Merhej \cite{Mer17}, who showed  if additionally $d=n-1$ and the unit normal vectors to the set have small BMO norm, then locally $X$ is contained in a bi-Lipschitz image of $\R^{n-1}$ (rather than just containing big pieces of $\R^{n-1}$ as in the definition of UR). She also has a higher codimensional version of this result \cite{Mer16}, which again requires some small oscillation of the tangents in the set $X$. 

There are other similar results for sets that inherit a Poincar\'{e} condition from some stronger topoligcal assumptions: G. C. David showed that any compact Ahlfors $d$-regular locally linearly contractible complete oriented topological $d$-manifold is UR \cite[Theorem 1.13]{Dav16}, and such spaces support a weak $(1,d)$-Poincar\'{e} inequality by \cite{Sem96} (see also \cite[Theorem 6.11]{HK98}). This is more general than our result in that it holds for non-Euclidean metric spaces, although the topological condition is more restrictive than being Loewner. 

The proof of the Main Theorem goes roughly as follows: the Poincar\'{e} inequality implies that there are many curves passing through the set by a result of Heinonen and Koskela. Using Dorronsoro's theorem, we can show that, for many $x\in X$ and $r>0$, and for any $(d-1)$-dimensional plane $V$, we can find parts of $X$ that lie close to a line segment passing through $x$ in $B(x,r)$ and have large angle from $V$. Inductively, this means we can actually find parts of $X$ close to $d$ many line segments passing through $x$ that have large angle from each other. We then use similar arguments to show that, for most balls on $X$, $X$ is approximately contained in a $d$-dimensional plane in those balls (otherwise, we could also find parts of $X$ close to a $(d+1)$st-line passing through each $x$, but we know $X$ is $d$-rectifiable and so it must be approximately $d$-flat somewhere, violating the existence of this extra line).  These two geometric properties imply that in fact $X$ is close to a $d$-dimensional plane in the Hausdorff metric, and this implies uniform rectifiability by a result of David and Semmes. We point out that this aspect of finding approximate line segments in many directions is  in a way reminiscent of how Bate finds Alberti representations in differentiability spaces \cite{Bat15}. \\

We would like to thank David Bate, Mihalis Mourgoglou, and Tatiana Toro for discussing this problem with him at various points in time, and Guy C. David for answering his questions about differentiability spaces while we were both at the 2018 conference `` The Geometric Measure Theory and its Connections" in Helsinki. We would also like to thank the anonymous referees for their patience with an atrocious first draft, comments that greatly improved the readability of the paper, for noting a mistake in citing and using Cheeger's paper, and also for providing clues to fix the mistake.

\def\Mod{{\rm Mod}}
%
%
%
%
%
%
%

\section{Preliminaries}

\subsection{Notation}

We will write $a\lesssim b $ if there is a constant $C>0$ so that $a\leq C b$,  $a\lesssim_{t} b$ if the constant depends on the parameter $t$, and $a\sim b$ and $a\sim_{t} b$ to mean $a\lesssim b \lesssim a$ and 
$a\lesssim_{t} b \lesssim_{t} a$ respectively. 

Given a metric space  $X$ with metric $d$, we will use Polish notation and write $d(x,y)=|x-y|$. Whenever $A,B\subset X$ we define
\[
\mbox{dist}(A,B)=\inf\{|x-y|;\, x\in A, \, y\in B\}, \, \mbox{and}\, \, \mbox{dist}(x,A)=\mbox{dist}(\{x\}, A). 
\]
Let $\diam A$ denote the diameter of $A$ defined as
\[
\diam A=\sup\{|x-y|;\, x,y\in A\}.
\]

We will also use $|A|$ to denote the measure of a set when the measure is clear from context. For example, if $A\subseteq \R$, then $|A|$ denotes the $1$-dimensional Lebesgue measure of $A$. 

We let $B(x,r)\subseteq X$ denote the {\it closed} ball centered at $x\in X$ of radius $r>0$, and if  $X=\R^{n}$ we will often write $\bB=B(0,1)$. If $B$ is a generic ball, we will write $x_B$ for its center and $r_{B}$ for its radius, so $B=B(x_{B},r_{B})$. 

We let $\cG(n,d)$ denote the Grassmannian, that is, the set of $d$-dimensional subspaces of $\R^{n}$ (that is, the $d$-dimensional planes containing the origin), and $\cA(n,d)$ denote the affine Grassmannian, which is the set of all $d$-dimensional planes in $\R^{n}$ (not necessarily containing the origin). 

Given a plane $V\in \cA(n,d)$, we let $\pi_V:\R^{n}\rightarrow V$ denote the projection into $V$, $V'\in \cG(n,d)$ the $d$-dimensional plane parallel to $V$ and containing the origin, and $V^\perp\in \cG(n-d,n)$ the orthogonal complement of $V'$. Given two planes $V,W\in \cG(n,d)$ with $\dim V\leq \dim W$, we let
\[
\angle(V,W) = |\pi_{W^{\perp}}|_{V}| = \sup_{x\in V\cap \bB}\dist(x,W).
\]
that is, $\angle(V,W)$ is the norm of the linear operator $\pi_{W^{\perp}}:V\rightarrow W^\perp$. Note that if $L$ is a line, then $\angle(L,W)$ is comparable to the usual angle between $L$ and $W$. If $V,W\in \cA(n,d)$, we let $\angle(V,W):=\angle(V',W')$. Note that from the above definition, if $\dim U \leq \dim V\leq \dim W$, then
\begin{equation}
\label{e:angle-triangle}
\angle(U,W) \leq \angle(U,V)+\angle(V,W).
\end{equation}

\subsection{Curves and Modulus}

In this section we introduce the notion of modulus of curve families. For a  more in depth treatment, see \cite{Heinonen} or \cite{Vuorinen}. Below $X$ will denote a complete metric space with metric $d$ and $\mu$ a {\it $C$-doubling measure}, meaning that for any ball $B\subseteq X$, if $2B$ is the ball with same center but twice the radius, then $\mu(2B)\leq C\mu(B)$. 

By a {\it curve} $\gamma$, we will mean any continuous image of a compact interval $I\subseteq X$. Given $\gamma$, we will denote this function also  as $\gamma :I\rightarrow X$. We define the length of $\gamma$ as
\[
\ell(\gamma) = \sup_{t_{1}<\cdots < t_{k}}\sum d(\gamma(t_{i}),\gamma(t_{i+1}))\]
where the supremum is over all sequences $a=t_{1}<\cdots < t_{k}=b$ if the endpoints of $I$ are $a$ and $b$. If $I$ is not closed, we define the length of $\gamma$ to be the supremum over the lengths of all subcurves with closed domain. If $\gamma$ is of finite length, we say $\gamma$ {\it rectifiable}, and then $\gamma$ factors as $\gamma=  \gamma_{s}\circ s_{\gamma}$ where $s_{\gamma}:I\rightarrow [0,\ell(\gamma)]$ and $\gamma_{s}$ is the arclength parametrization, that is, a $1$-Lipschitz function $\gamma_{s}: [0,\ell(\gamma)]\rightarrow X$ with $\ell(\gamma_s|_{[0,t]})=t$. We will assume all rectifiable curves below are arclength parametrized. If all closed subcurves are rectifiable, we say $\gamma$ is {\it locally rectifiable}. 

Given a metric space $X$, a Borel measure $\mu$, a family of curves $\Gamma$ in $X$, and a Borel function $\rho$, we say $\rho$ is {\it admissible} for $\Gamma$ if for each locally rectifiable curve $\gamma\in \Gamma$,
\[
\int_{\gamma} \rho : = \int_{0}^{\ell(\gamma)} \rho\circ\gamma \geq 1  \;\; \mbox{for all} \;\; \gamma\in \Gamma.\]

Note that this notation means we are integrating $\rho$ composed with {\it the function} $\gamma$ and not $\rho$ on the {\it image} of $\gamma$. However, the former is at least the latter: since the arclength parametrization is $1$-Lipschitz, $\cH^{1}(\gamma(A))\leq |A|$ for any $A\subseteq [0,\ell(\gamma)]$ (see \cite[Theorem 7.5]{Mattila}), and so 
\begin{align*}
\int_{\gamma} \rho  
& =\int_{0}^{\infty} |\{t\in [0,\ell(\gamma)]: \rho\circ \gamma(t) >\lambda  \}|d\lambda \\
& =\int_{0}^{\infty} |\gamma^{-1}(\{x\in \gamma: \rho(x)>\lambda \})|d\lambda \\
& \geq \int_{0}^{\infty} \cH^{1}(\{x\in \gamma: \rho(x)>\lambda \})|d\lambda 
 =\int_{\gamma} \rho d \cH^{1}
\end{align*}
although these two integrals may not equal, for example if $\gamma$ doubles back on itself. If $\gamma$ is only locally rectifiable, we define $\int_{\gamma}\rho$ to be the supremum of $\int_{\gamma'}\rho$ over all rectifiable subcurves $\gamma'$.

We define the {\it $p$-modulus} of $\Gamma$ to be
\[
\inf \ck{\int \rho^{p}d\mu: \;\; \rho\mbox{ admissible for $\Gamma$}}.
\]

We say $(X,\mu)$ is a {\it $p$-Loewner space} if, whenever $E,F\subseteq X$ are two disjoint continua, and $\Gamma(E,F)$ is the collection of curves in $X$ starting in $E$ and ending in $F$, then
\[
\Mod_{p}(\Gamma(E,F))\gec_{t} 1 \;\;\mbox{ whenever }\;\; \Delta(E,F):=\frac{\dist(E,F)}{\min\{\diam E,\diam F\}}\leq t.
\]

\subsection{Ahlfors regular spaces}

The results in this section are also about modulus, but are specific to Ahlfors regular spaces. In particular, $X$ will now denote a complete metric space that is the support of an Ahlfors $d$-regular measure $\mu$.

The following lemma is standard, but we give a proof for completeness. 
\begin{lemma}
\label{l:short-curves}
Let $X$ be an Ahlfors $d$-regular Loewner space with measure $\mu$, $d\geq 2$, $B$ be a ball in $X$ and $E,F\subseteq B$ two disjoint continua so that $\Delta(E,F)\leq t$. Let $\Gamma_{C,B}(E,F)$ be those curves in $\Gamma(E,F)$ of length at most $Cr_{B}$. Then for $C$ large enough (depending on $t,d,\mu$, and the Loewner constants),
\begin{equation}
\label{e:loewner}
\Mod_{d}(\Gamma_{C,B}(E,F))\gec_{C,t} 1.
\end{equation}
\end{lemma}

\begin{proof}
Recall that there is a constant $C_0$ depending on $t$ and the Loewner constants so that 
\[
\Mod_{d}(\Gamma(E,F))\geq C_0.
\]
Let $\Gamma_{1}$ be those curves in $\Gamma(E,F)$ that contain a point outside of $AB$ for some $A\geq 2$ to be chosen shortly,  and $\Gamma_{2}$ be those curves in $\Gamma(E,F)$ contained in $AB$ but so that their length is at least $Cr_{B}$. Observe that 
\[
\Gamma(E,F)\subseteq \Gamma_{C,B}(E,F)  \cup \Gamma_1\cup \Gamma_2.
\]

We claim that for $A$ large enough and $C$ large enough depending on $A$ (and each depending on $C_0,d$, and the Ahlfors regularity), 

\begin{equation}
\label{e:gamma1gamma2}
\max_{i=1,2}\Mod_{d}(\Gamma_{i})<\frac{C_0}{4}. 
\end{equation}
If we show this, then by the above containment and the subadditivity of the modulus (see \cite[Equation (7.7)]{Heinonen}), 
\[
\Mod_{d} (\Gamma_{C,B}(E,F))
\geq \Mod_{d} (\Gamma(E,F))-\Mod_{d} (\Gamma_1)-\Mod_{d} (\Gamma_2)>\frac{C_{0}}{2}.
\]
which proves the lemma. So now we focus on \eqref{e:gamma1gamma2} and start by estimating $\Mod_{d}(\Gamma_{1})$. Let 
\[
\rho_{1}(x) = \frac{1}{r_{B}+|x-x_{B}|} \frac{1}{\log A}\one_{AB}
\]
Then it is not hard to show that $\int_{\gamma} \rho_{1}\gec 1$ for all $\gamma\in \Gamma_1$. Thus,
\[
\Mod_{d}(\Gamma_{1})
\leq \int \rho_1^{d}d\mu 
\lec (\log A)^{1-d}.
\]
See \cite[Theorem 7.18]{Heinonen} for a proof of a similar estimate. Since $d\geq 2$, we can choose $A$ large enough (depending on $d$ and $C_0$) so that 
\[
\Mod_{d}(\Gamma_{1})<\frac{C_0}{4}.
\]
 
Now we bound $\Mod_{d}(\Gamma_{2})$.  Let $\rho_2 = \frac{1}{Cr_{B}}\one_{AB}$. Then $\rho_2$ is admissible for $\Gamma_{2}$, and so
\[
\Mod_{d}(\Gamma_2)
\leq \int \rho_2^{d} 
\lec (Ar_{B})^{d} (Cr_{B})^{-d} = \frac{A^{d}}{C^{d}}.
\]
Hence, we can pick $C$ depending on $A$ and $C_0$ (so just really on $C_0$) so that 
\[
\Mod_{d}(\Gamma_2)<\frac{C_0}{4}.
\]
This proves \eqref{e:gamma1gamma2}, and thus finishes the proof.

\end{proof}

The connection between the Poincar\'{e} inequality an Loewner spaces is via the following result.

\begin{theorem}
\label{t:HK}
A complete, proper, Ahlfors $d$-regular metric measure space  $(X,\mu)$ admits a weak $(1,d)$-Poincar\'{e} inequality if and only if it is a $d$-Loewner space. The constants in the definition of the weak $(1,d)$-Poincar\'{e} depend on the Ahlfors regularity constant and the Loewner constants implicit in \eqref{e:loewner}, and conversely the constants in \eqref{e:loewner} depend on the Poincar\'{e} constants.
\end{theorem}

This follows from \cite[Theorems 5.7 and 5.12]{HK98}. Note that the first of these theorems (the forward implication) requires $X$ to be $\phi$-convex; we won't define this, but it is satisfied when $X$ is quasiconvex, which holds when $X$ has a weak $(1,d)$-Poincar\'{e} inequality by a theorem of Semmes (see the appendices of \cite{Che99,KM16-primer}).

\subsection{Christ-David Cubes}

We recall the following version of ``dyadic cubes" for metric spaces, first introduced by David \cite{Dav88} for Ahflors regular sets, but generalized in \cite{Chr90} and \cite{HM12}.

 \begin{theorem}
Let $X$ be a doubling metric space. Let $X_{k}$ be a nested sequence of maximal $\rho^{k}$-nets for $X$ where $\rho<1/1000$ and let $c_{0}=1/500$. For each $k\in\bZ$ there is a collection $\cD_{k}$ of ``cubes,'' which are Borel subsets of $X$ such that the following hold.
\begin{enumerate}
\item For every integer $k$, $X=\bigcup_{Q\in \cD_{k}}Q$.
\item If $Q,Q'\in \cD=\bigcup \cD_{k}$ and $Q\cap Q'\neq\emptyset$, then $Q\subseteq Q'$ or $Q'\subseteq Q$.
\item For $Q\in \cD$, let $k(Q)$ be the unique integer so that $Q\in \cD_{k}$ and set $\ell(Q)=5\rho^{k(Q)}$. Then there is $\zeta_{Q}\in X_{k}$ so that
\begin{equation}\label{e:containment}
B_{X}(\zeta_{Q},c_{0}\ell(Q) )\subseteq Q\subseteq B_{X}(\zeta_{Q},\ell(Q))
\end{equation}
and $ X_{k}=\{\zeta_{Q}: Q\in \cD_{k}\}$.
\end{enumerate}
\label{t:Christ}
\end{theorem}

\subsection{$\beta$-numbers and flat balls in Euclidean Loewner spaces}

The objective of this section is to introduce $\beta$-numbers and gather together a few lemmas about them we will need later. The most important of these is that, given an Ahlfors $d$-regular subspace $(X,\mu)$ of $\R^{n}$ and a ball $B$ centered on $X$, we can find a large ball contained in $B$ where $X$ is approximately flat (how large the ball is will depend on how flat we would like $X$ to be in the ball). To quantify flatness, we use Jones' $\beta$-numbers.

Let  $X\subseteq \R^{n}$ be as above. For $V\in \cA(n,d)$, $x\in X$ and $r>0$, let
\[
\beta_{X}(x,r,V) =  \sup_{y\in B(x,r)\cap X} \frac{\dist(y,V)}{r}, \;\; \beta_{X}(x,r) = \inf_{V\in \cA(n,d)} \beta_{X}(x,r,V).
\]
Given a ball $B(x,r)$ centered on $X$, we will also sometimes write $\beta_{X}(B(x,r))$ for $\beta_{X}(x,r)$. It is not hard to show that, if $B(x,r)\subseteq B(y,s)$ are centered on $X$, then 
\begin{equation}
\label{e:betamonotone}
\beta_{X}(x,r,V)\leq \frac{s}{r} \beta_{X}(x,s,V).
\end{equation}

The objective of this section is to prove the following lemma.

\begin{lemma}
\label{l:flat-ball}
Let $(X,\mu)$ be an Ahlfors $d$-regular $d$-Loewner subspace of $\R^{n}$ and $d\geq 2$. For all $\ve\in (0,1/2)$, $x\in X$, and $0<r<\diam X$, there is $r'\gec_{\ve} r$ and $x'\in B(x,r/2)\cap X$ so that
\[
\beta_{X}(x',r')<\ve.\]
\end{lemma}

To prove this, we need to review some results about Hausdorff convergence. \\

Recall that a sequence of compact sets $X_j$ converge to another compact set $X$ in the Hausdorff metric in $\R^{n}$ if 
\[
\lim_{j\rightarrow\infty} \max\ck{\sup_{x\in X} \dist(x,X_{j}) ,\sup_{x\in X_j} \dist(x,X)} =0.
\]
Given closed nonempty but possibly unbounded sets $X_j$ and $X$ in $\R^{n}$, we will say $X_{j}\rightarrow X$ in the Hausdorff metric if  for each $R>0$ there is $\ve_{j}\downarrow 0$ so that $X_j\cap B(0,R+\ve_{j})$ converges to $X\cap B(0,R)$ in the Hausdorff metric, or equivalently, if 
\[
\lim_{j\rightarrow\infty} \max\ck{\sup_{x\in X\cap B(0,R)}\dist(x,X_{j}), \sup_{x\in X_j\cap B(0,R)}\dist(x,X)}=0 \;\; \mbox{ $\forall$ }R>0.
\]

\begin{lemma}
\label{l:ahlforshaus}
Let $\mu_j$ be a sequence of uniformly Ahlfors $d$-regular measures in $\R^{n}$, $X_j=\supp \mu_j$, and suppose $0\in X_j$ for all $j$. Suppose also that $\inf \diam X_j>0$. Then we may pass to a subsequence so that $X_j$ converges in the Hausdorff metric to a closed set $X$, $\mu_j$ converges weakly to an Ahlfors $d$-regular measure $\mu$ with the same constants and $\supp \mu=X$. 
\end{lemma}

The proof is not too bad, we just give some hints: first, $\mu_j|_{B(0,r)}$ is uniformly bounded by Ahlfors regularity for all $r>0$, so by a diagonalization argument, we may pass to a subsequence so that $\mu_j$ converges weakly insited every ball centered at 0 (and hence everywhere) to a Radon measure $\mu$. By testing against bump-functions, one can show $\mu$ is also Ahlfors $d$-regular with the same constant as the $\mu_j$. If $X=\supp \mu$, then it is not hard from here to use the weak convergence of these two measures to show $X_j\rightarrow X$ in Hausdorff distance.\\

Below is a compactness lemma we will need.

\begin{lemma}
\label{l:limlem}
Let $d\geq 2$, $\mu_{j}$ be an Ahlfors $d$-regular measure in $\R^{n}$, and  $X_j=\supp \mu_{j}$ be so that $(X_j,\mu_j)$ admits a weak $(1,p)$-Poincar\'{e} inequality for some $p>1$ with the same constants for all $j$, and suppose $0\in X_j$. Then there is a subsequence that converges in the Hausdorff distance to an Ahlfors $d$-regular set also satisfying a weak $(1,p)$-Poincar\'{e} inequality. 
\end{lemma}

%
%

\begin{proof}
By Lemma \ref{l:ahlforshaus}, we may pass to a subsequence so that $X_j$ converges in the Hausdorff metric to a set $X\subseteq \R^{n}$ and so that $\mu_j$ conveges weakly to an Ahlfors $d$-regular measure $\mu$ supported on $X$. Thus, the tuples $(X_{j},|\cdot|,0,\cH^{d}|_{X_j})$ forms a sequence of uniformly Ahlfors $d$-regular complete pointed metric spaces converging in the measured Gromov-Hausdorff sense to $(X,|\cdot|,0,\cH^{d}|_{X_j})$ (this is a lot of terminology to unpack, so we instead refer the reader to \cite[Section 2]{Kei03} as a reference). By \cite[Theorem 3]{Kei03}, $(X,\mu)$ also satisfies the weak $p$-Poincar\'{e} inequality with constants depending on the uniform Poincar\'{e} constants for the $(X_j,\mu_j)$. 
\end{proof}

\begin{proof}[Proof of Lemma \ref{l:flat-ball}]
It suffices to prove the lemma in the case that $x=0$ and $r=1$. Suppose there was $\ve\in (0,1/4)$ and a sequence of Ahlfors $d$-recular $d$-Loewner spaces $(X_{j},\mu_j)$ in  $\R^{n}$ with the same constants so that $\diam X_j'\geq 1$, $0\in X_j$, and for all $x'\in B(0,1/2)\cap X_j$ and $r'\geq 1/j$, 
\[
\beta_{X_j}(x',r')\geq \ve.
\]
These spaces satisfy a weak $(1,d)$-Poincar\'{e} inequality with the same constants for all $j$ by Theorem \ref{t:HK}. We can pass to a subsequence so that they converge in the Hausdorff metric to another $d$-regular set $X$ satifsying a weak $(1,d)$-Poincar\'{e} inequality and $\mu_{j}\warrow \mu$ for some Ahlfors $d$-regular measure $\mu$. By Cheeger's theorem and Theorem \ref{t:rectifiable}, $(X,\mu)$ is $d$-rectifiable. Since $\mu\ll\cH^{d}|_{X}\ll \mu$ by Ahlfors $d$-regularity, this means $X$ is $d$-rectifiable. In particular, $X$ has a tangent at some point $x\in X\cap B(0,1/2)$ (see the discussion after \cite[Theorem 1.1]{Vil17} and \cite[Section 3]{Vil17}), so there is a plane $P$ passing through $x$ and $r>0$ small so that 
\begin{equation}
\label{e:betaxrP<e/4}
\beta_{X}(x,r,P)<\ve/4.
\end{equation}
There is $\ve_{j}$ so that $X_j\cap B(0,1+\ve_{j})$ converges to $X\cap B(0,1)$ in the Hausdorf metric, so for $j$ large enough, 
\[
\sup_{x'\in X\cap B(x,r)} \dist(x',X_j) + \sup_{x'\in X_{j}\cap B(x,r)} \dist(x',X)<\frac{\ve r}{4}.
\]
 In particular, for $j$ large enough we can find $x_j\in X_j\cap B(x,\ve r/4)$ and for each $y'\in B(x_{j},r/2)\cap X_{j}\subseteq B(x,r)\cap X_j$, there is $y\in X$ with $|y-y'|<\frac{\ve r}{4}$. For each such $y$, 
 \[
 |y-x|\leq |y-y'|+|y'-x_j|+|x_j-x|
 \leq \frac{\ve r}{4}+\frac{r}{2}+ \frac{\ve r}{4}<r
\]
so $y\in B(x,r)\cap X$. Thus,
\[
\dist(y',P)
\leq |y'-y|+\dist(y,P)
\stackrel{\eqref{e:betaxrP<e/4}}{<} \frac{\ve r}{4} + \frac{\ve r}{4} = \frac{\ve r}{2}.
\]
If we take the supremum over all $y'\in B(x_{j},r/2)\cap X_{j}$, then for $1/j<r/2$, by how we chose the $X_j$.
\[
\ve\leq \beta_{X_{j}}(x_j,r/2)<\ve,
\]
which is a contradiction.

\end{proof}

\section{Proof of the Main Theorem}

From now on we let $(X,\mu)$ denote a closed Ahlfors $d$-regular Loewner space in $\R^{n}$. We can assume without loss of generality that $\mu=\cH^{d}|_{X}$ and just say $X$ is a closed Ahlfors $d$-regular Loewner space in $\R^{n}$. We will frequently denote the $\mu$-measure of a set $A\subseteq X$ by $\mu(A)=\cH^{d}(X\cap A)=|A|$. 

We will assume all implied constants that appear below depend on $n$, $d$, and also on the Poincar\'{e} and Ahlfors regularity constants for $X$, and hence write $\sim$ instead of $\sim_{d,n,A,C,\lambda}$.

The Main Theorem will follow from three main lemmas. The first lemma is a black-box theorem due to David and Semmes  \cite[Theorem I.2.4]{of-and-on}. First we need some new notation.

For $x\in X$, $r>0$ and $V\in \cA(n,d)$, we define the $d$-dimensional bilateral $\beta$-number  with respect to $V$ to be 
\[b\beta_{X}(x,r,V)=r^{-1}\left(\sup_{y\in
X\cap B(x,r)}\dist(y,V)+ \sup_{y\in V\cap
B(x,r)}\dist(y,X)\right)\] 
and then define
\[b\beta_{X}(x,r)=\inf_{V\in \cA(n,d)}b\beta_{X}(x,r,V).\]

\begin{lemma}[The Bilateral Weak Geomeric Lemma (BWGL)]
\label{l:BWGL}
An Ahlfors $d$-regular set $E\subseteq \R^{n}$ is UR if and only if, for each $\ve>0$ and $R\in \cD$ (where $\cD$ are the Christ-David cubes for $E$)
\begin{equation}
\label{e:BWGL}
\sum_{Q\subseteq R \atop b\beta_{E}(2B_{Q})\geq \ve}|Q|\lec_{\ve} |R|.
\end{equation}
\end{lemma}

Thus, the Main Theorem will follow once we show $X$ satisfies the BWGL. This will follow from the next two lemmas which will be the focus of the paper. 

\begin{lemma}[Weak Geometric Lemma (WGL)] 
\label{l:wgl}
Let $X$ be a closed Ahlfors $d$-regular Loewner space in $\R^{n}$. For $R\in \cD$ and $\delta>0$,
\[
\sum_{Q\subseteq R \atop \beta_{X}(B_{Q})\geq \delta}|Q|
\lec |R|.
\]
\end{lemma}

This is the so-called {\it weak geometric lemma} (WGL) in the argot of David and Semmes \cite[Chapter 5]{DS}. This property alone does not imply UR (see \cite[Section 20]{DS}), but it will when coupled it with the following result:

\begin{lemma}[Many Segments Property (MS)]
\label{l:d-lines}
Let $X$ be a closed Ahlfors $d$-regular Loewner space in $\R^{n}$. For $x\in X$, $r>0$, and $\theta>0$, define 
\[
\eta_{X}^{\theta}(x,r)=\inf_{L_{1},...,L_{d}} \sup_{y\in (L_{1}\cup\cdots \cup L_{d})\cap B(x,r)} \frac{\dist(y,X)}{r}
\]
where the infimum is over all lines $L_{1},...,L_{d}$ passing through $x$ so that $\angle(L_{k+1},U_{k})\geq \theta/2$ for $k=1,2,..,d-1$ where $U_k$ is the $k$-dimensional plane spanned by the lines $L_{1},...,L_k$ , and set
\[
\eta_{X}^{\theta}(Q) =  \sup_{x\in Q}\eta_{X}^{\theta}(x,\ell(Q)).
\]
There is $\theta\gec 1$ so that for all $\delta>0$, we have
\begin{equation}
\label{e:manysegments}
\sum_{Q\subseteq R\atop \eta_{X}^{\theta}(Q)\geq \delta}|Q|\lec_{\delta} |R| \;\;\; \mbox{ for all }R\in \cD.
\end{equation}
\end{lemma}

We will say an Ahlfors $d$-regular set satisfies the {\it Many Segments Property} (MS) if it satisfies \eqref{e:manysegments} for any $\delta>0$. We do not know if this property is sufficient for UR, but as mentioned earlier, it does when coupled with the WGL, as we'll show now:

\begin{lemma}
If $X,\mu$ and $\theta$ be as in Lemma \ref{l:d-lines}, then for all $\delta>0$ there are $M>1$ and $\ve>0$ so that if $B$ is a ball centered on $X$ with $r_{B}<M^{-1}\diam X$, $\beta_{X}(MB)<\ve$, and $\sup_{x\in B\cap X} \eta_{X}^{\theta}(x,Mr_{B})<\ve$, then $b\beta_{X}(B)<\delta$. 
\end{lemma}

\begin{proof}
Let $\theta$ be as in Lemma \ref{l:d-lines} and $\delta>0$. Without loss of generality, we can assume $B=\bB$. Suppose there is $\delta>0$ so that instead that for all $j$ we could find Ahlfors $d$-regular sets $X_j\subseteq \R^{n}$ (with the same constants) containing $0$ so that $\diam X_j\geq j$, $\beta_{X}(j\bB)<\frac{1}{j^2}$ and $\sup_{x\in \bB\cap X_j} \eta_{X}^{\theta}(x,j)<\frac{1}{j^2}$, but $b\beta_{X_{j}}(\bB)\geq \delta$. We can pass to a subsequence so that this converges in the Hausdorff metric to an Ahlfors $d$-regular set $X$ containing $0$ and with the property that for all $x\in X$ there are $d$ lines $L_{1}(x),...,L_{d}(x)\subseteq X$ containing $x$ so that the angles between $L_{k+1}(x)$ and the span of $L_{1}(x),...,L_{k}(x)$ are at least $\theta/2>0$, and so that $\beta_{X}(r\bB)=0$ for all $r>0$. In particular, $X$ is contained in a $d$-dimensional plane, which we can assume without loss of generality to be $\R^{d}$. Moreover, $b\beta_{X}(\bB)\geq \delta$. Since $X\subseteq \R^{d}$, this implies there is $z\in \bB\cap \R^{d}$ with $\dist(z,X)\geq \delta$. Let $\delta'=\dist(z,X)$ and $x\in \d B(z,\delta' )\cap X$. If $V$ is the $(d-1)$-dimensional plane in $\R^{d}$ tangent to $B(z,\delta')\cap \R^{d}$ at $x$, then there is at least one $i$ so that $L_{i}(x)$ is not parallel with $V$, so in particular, 
\[
\emptyset \neq L_{i}(x)\cap  B(z,\delta')^{\circ}
\subseteq L_{i}(x)\cap \R^{d}\backslash X\]
(where we are taking the interior with respect to $\R^{d}$), whereas we know $L_{i}(x)\subseteq X$, so we get a contradiction. 
\end{proof}

\begin{corollary}
For $\theta$ is as in Lemma \ref{l:d-lines}, and for all $\delta>0$, there is $N\in\mathbb{N}$ and $\ve>0$ so that if $Q\in \cD$, $Q^{N}$ is the $N$-th generation ancestor of $Q$, $\beta_{X}(B_{Q^{N}})<\ve$, and $\eta_{X}^{\theta}(Q^{N})<\ve$, then $b\beta_{X}(B_{Q})<\delta$. 
\end{corollary}

This follows from the previous lemma and the fact that $\beta_{X}(B_{Q^N})\lec_{N} \beta_{X}(B_{Q})$ and  $\eta_{X}^{\theta}(B_{Q^N})\lec_{N} \eta_{X}^{\theta}(B_{Q})$, we leave the details to the reader. 

We now finish the proof of the Main Theorem using the above results. Observe that the map sending $Q\rightarrow Q^{N}$ is at most $C(N)$-to-$1$, and so by the previous Corollary, 
\begin{align*}
\sum_{Q\subseteq R\atop b\beta_{X}(B_{Q})\geq \delta}|Q|
& \leq\sum_{Q\subseteq R\atop \beta_{X}(B_{Q^{N}})\geq \ve}|Q|
+\sum_{Q\subseteq R\atop \eta_{X}^{\theta}(Q^{N})\geq \ve}|Q|
 \lec \sum_{Q\subseteq R^{N}\atop \beta_{X}(B_{Q})\geq \ve } |Q|
+\sum_{Q\subseteq R^{N}\atop \eta_{X}^{\theta}(B_{Q})\geq \ve } |Q|
\\
& \lec |R^{N}|\lec |R|.
\end{align*}

Thus, by the BWGL (Lemma \ref{l:BWGL}), $X$ is UR.\\

This finishes the proof of the Main Theorem  assuming the WGL and MS properties (i.e. Lemmas \ref{l:wgl} and \ref{l:d-lines}). The remainder of the paper focuses on proving these two results, each of which will be the focus of the next two sections respectively.

\section{Proof of MS}

In this section we focus on proving Lemma \ref{l:d-lines}. To show that most points in $X$ have $d$-many segments pointing in a linearly independent set of directions lying close to $X$, we will reduce this to showing that, through most  $x\in X$ and any $(d-1)$-dimensional plane $V$ passing through $x$, we can fine just one segment transversal to $V$ that lies close to $X$. We will then use this repeatedly to build up a set of $d$-many independent segments. We show explicitly the reduction in the following subsection. 

\subsection{$\eta$-numbers}
For $x\in X$, $r>0$, $V\in \cA(n,k)$ with $1\leq k\leq n-1$, and $\theta>0$, define
\[
\eta_{X}^{V,\theta}(x,r) = \inf_{L}\sup_{y\in B(x,r)\cap L}\frac{\dist(y,X)}{r}
\]
where the infimum is over all lines $L$ passing through $x$ so that if $e_{L}$ is the vector parallel to $L$, and $V'\in \cG(n,k)$ is parallel to $V$, then 
\[
\angle(L,V)=\dist(e_{L},V')\geq \theta.
\]

Note that the quantity is unchanged if we instead require our planes to be in $\cG(n,k)$ for some $k$, since if $V'$ is the plane through the origin parallel to $V$, then $\eta_{X}^{V,\theta}=\eta_{X}^{V',\theta}$, that is, the definition only compares angles between planes and not their position. We prefer to allow for affine planes as a matter of convenience below.

We record a few basic properties of the $\eta$-numbers. Firstly, since all lines pass through $x$ in this definition, we immediately have 
\[
0\leq \eta_{X}^{V,\theta}(x,r)\leq 1.
\]

\begin{lemma}
Let $x\in X$, $V\in \cA(n,k)$ with $1\leq k<n$, and $\theta>0$. Then 
\begin{equation}
\label{e:eta-monotone}
\eta_{X}^{V,\theta}(x,r)\leq \frac{s}{r} \eta_{X}^{V,\theta}(x,s) \;\; \mbox{ for all } 0<r\leq s.
\end{equation}
If $x,y\in X$, then
\begin{equation}
\label{e:eta-lip}
\eta_{X}^{V,\theta}(x,r)\leq \eta_{X}^{V,\theta}(y,r)+\frac{|x-y|}{r}.
\end{equation}
Finally,
\begin{equation}
\label{e:theta-monotone}
\eta_{X}^{V,\theta}(x,r)\leq \eta_{X}^{V,\theta'}(x,r) \;\; \mbox{ for }\;\; \theta \leq \theta'.
\end{equation}
\end{lemma}

\begin{proof}
For \eqref{e:eta-monotone}, let $L$ be any line passing through $x$. Then
\[
r\eta_{X}^{V,\theta}(x,r)
\leq \sup_{y\in B(x,r)\cap L}   \dist(y,X) 
 \leq \sup_{y\in B(x,s)\cap L}   \dist(y,X) \]
 and infimizing over all $L$, we obtain $r\eta_{X}^{V,\theta}(x,r)<s\eta_{X}^{V,\theta}(x,s)$. \\
 
 For \eqref{e:eta-lip}, let $L$ be the line that infimizes $\eta_{X}^{V,\theta}(x,r)$. Let $L'=L+y-x$. Then $L'$ passes through $y$ and also has angle at least $\theta$ with $V$. If $z'\in L'\cap B(y,r)$, then $z:=z'-y+x\in B(x,r)\cap L$, and so there is $z''\in X$ with $|z-z''|\leq \eta_{X}^{V,\theta}(x,r)r$. Thus,
 \[
 \dist(z',X)\leq |z'-z''|\leq  |z'-z|+|z-z''|
 \leq |x-y|+\eta_{X}^{V,\theta}(x,r)r.\]
 Dividing both sides by $r$ and taking the supremum over all $z'\in B(y,r)\cap L'$ gives \eqref{e:eta-lip}.
 
 Finally, to prove \eqref{e:theta-monotone}, let $L$ be a plane infimizing $\eta_{X}^{V,\theta'}(x,r)$, then it has angle at least $\theta'$ from $V$, and since $\theta'\geq \theta$, it also has angle at least $\theta$ from $V$, and so 
 \[
 \eta_{X}^{V,\theta}(x,r)\leq \sup_{y\in B(x,r)\cap L}\frac{\dist(y,X)}{r}
 = \eta_{X}^{V,\theta'}(x,r),
 \]
 which proves \eqref{e:theta-monotone}.

\end{proof}

\subsection{Finding transversal segments}

The main objective of this section is the following lemma.

\begin{lemma}
\label{l:eta-carleson}
Let $V\in \cA(n,d-1)$. For $Q\in \cD$, let
\[
\eta_{X}^{V,\theta}(Q) =  \sup_{x\in Q}\eta_{X}^{V,\theta}(x,\ell(Q)).
\]
There is $\theta\gec 1$ so that for all $\delta>0$, 
\begin{equation}
\label{e:nxvtheta-carleson}
\sum_{Q\subseteq R\atop \eta_{X}^{V,\theta}(Q)\geq \delta}|Q|\lec_{\delta} |R| \;\;\; \mbox{ for all }R\in \cD.
\end{equation}
\end{lemma}

Let us first show how this lemma implies Lemma \ref{l:d-lines}:

\begin{proof}[Proof of Lemma \ref{l:d-lines}]
Let $\cV$ be a maximally $\frac{\theta}{2}$-separated set in $G(n,d-1)$ with respect to the distance
\[
d(V,U):= \angle(V,U).
\]
We will use the following property repeatedly below: for each $U\in \cG(n,k)$ and $k\leq d-1$, we can contain $U$ in a plane $U'\in \cG(n,d-1)$ and thus we can always find $V\in \cV$ so that 
\[
\angle(U,V)\leq \angle(U',V)<\frac{\theta}{2}.
\]

Let
\[
\cB_{\delta}^{V} = \{Q\in \cD: \eta_{X}^{V,\theta}(Q)\geq \delta\}, \;\; \cB_{\delta} =\bigcup_{V\in \cV} \cB_{\delta}^{V},\;\; \cG = \cD\backslash \cB_{\delta}. 
\]
Let $x\in Q\in \cG$ and pick $V_0\in \cV$. Without loss of generality, we will assume $x=0$. Then $\eta_{X}^{V_{0},\theta}(Q)<\delta$, and so there is a line $L_{1}(x)$ passing through $x$ so that 
\[
\sup_{y\in L_{1}(x)\cap B(x,\ell(Q))}\frac{\dist(y,X)}{\ell(Q)}<\delta. 
\]
By definition of $\cV$, there is $V_1\in \cV$ so that $\angle(L_{1}(x),V_{1})<\frac{\theta}{2}$. Then $Q\in \cG$ implies $\eta_{X}^{V_{1},\theta}(Q)<\delta$, and so there is a line $L_{2}(x)$ passing through $x$ so that $\angle(L_{2}(x),V_{1})\geq \theta$, hence
\[
\angle(L_{2}(x),L_{1}(x))
 \stackrel{\eqref{e:angle-triangle}}{\geq}  \angle(L_{2}(x),V_{1})-\angle(L_{1}(x),V_{1}) \geq \frac{\theta}{2}.
\]
and
\[
\sup_{y\in L_{2}(x)\cap B(x,\ell(Q))}\frac{\dist(y,X)}{\ell(Q)}<\delta. 
\]
Inductively, for $2\leq k\leq d-1$, if $U_{k}$ is the plane spanned by the lines $L_{1}(x),\cdots , L_{k}(x)$, then  we can find $V_{k}\in \cV$ so that  $\angle(U_{k},V_{k})<\frac{\theta}{2}$. Then $Q\in \cG$ implies implies $\eta_{X}^{V_{k},\theta}(Q)<\delta$, and so there is a line $L_{k+1}(x)$ passing through $x$ so that $\angle(L_{k+1}(x),V_{k})\geq \theta$, and so
\begin{equation}
\label{e:angleLU}
\angle(L_{k+1}(x),U_{k})\geq \angle(L_{k+1}(x),V_{k})-\angle(U_{k},V_{k}) \geq \frac{\theta}{2}
\end{equation}
and
\[
\sup_{y\in L_{k+1}(x)\cap B(x,\ell(Q))}\frac{\dist(y,X)}{\ell(Q)}<\delta. 
\]
By induction, we can find lines $L_{1}(x),...,L_{d}(x)$ satisfying \eqref{e:angleLU} for all $k<d$, which implies $\eta_{X}^{\theta}(Q)<\delta$ for all $Q\in \cG$, and 
\[
\sum_{Q\subseteq R\atop \eta_{X}^{\theta}(Q)\geq \delta}|Q|
=\sum_{Q\in \cB_{\delta}\atop Q\subseteq R} |Q|
\leq \sum_{V\in \cV} \sum_{Q\in \cB_{\delta}^{V}\atop Q\subseteq R}|Q|
\stackrel{\eqref{e:nxvtheta-carleson}}{\lec} |R|
\]
where in the last inequality we used Lemma \ref{l:eta-carleson} and the fact that $|\cV|\lec_{\theta,n}1$.

\end{proof}

Thus, it remains to prove Lemma \ref{l:eta-carleson}, which we prove in the next two subsections.

\subsection{Proof of Lemma \ref{l:eta-carleson}: Part I}

We first prove a lemma that says, inside any ball, we can find a large subset, for each point of which we have nice estimates on the $\eta$-numbers.

\begin{lemma}
\label{l:generalV}
Let $C$ be as in Lemma \ref{l:short-curves} and let $c\in (0,1)$. There is $\theta>0$ depending on $c$, the Ahlfors regularity, and the Poincar\'{e} constants, so that the following holds. Let $V\in \cG(n,k)$ for some $1\leq k\leq n-1$. Let $B$ be a ball centered on $X$ with $0<r_{B}<\diam X$ and suppose there is $x_0\in \frac{1}{2} B\cap X$ with $\dist(x_0,V+x_{B})\geq c r_{B}$. Then there is $E_{B}^{V}\subseteq CB\cap X$ so that $|E_{B}^{V}|\gec_{c} |CB|$ and 
\begin{equation}
\label{e:etaintonE}
\int_{0}^{r_{B}} \eta_{X}^{V,\theta}(x,r)^2\frac{dr}{r}\lec_{c} 1 \mbox{ for all }x\in E_{B}.
\end{equation}
\end{lemma}

Here we are abusing notation and denoting $|B| = \cH^{d}(B\cap X)$.

%

Some of the ideas for this proof come from \cite{Jon88}, \cite[Section III.4]{Dav91}, and \cite{JKV97}. 

\begin{proof}
Without loss of generality, we can just assume $V\in \cG(n,n-1)$, and this will imply the general case (since we can contain any $V$ in a $(n-1)$-dimensional plane so that the assumptions of the lemma still hold). Let $\theta>0$ to be decided later. For convenience, we will write $\eta=\eta_{X}^{V,\theta}$ below.

Let $A_{1}\subseteq \frac{c}{4} B\cap X$ and $A_{2}\subseteq B(x_{0},\frac{c}{4}r_{B})\cap X$ be two continua of diameter at least $\frac{c}{4} r_{B}$ (which exist since $X$ is connected and $\diam X\geq r_{B}$). 

For technical reasons, it will be more convenient to work with loops rather than curves. For a family of rectifiable curves $\Gamma_0$ contained in $X$ and $\gamma\in \Gamma_0$, let $\dot{\gamma}=\gamma$ on $[0,\ell(\gamma)]$ and $\dot{\gamma}(t) = \gamma(2\ell(\gamma)-t)$ for $t\in [\ell(\gamma),2\ell(\gamma)]$. Then $\dot{\gamma}$ is a curve of length $2\ell(\gamma)$ (although its image is the same as $\gamma$). Let $\dot{\Gamma}_{0} = \{\dot{\gamma}:\gamma\in \Gamma_{0}\}$. Then one can show
\[
\Mod_{d}(\dot{\Gamma}_{0}) = 2^{-d} \Mod_{d}({\Gamma}_{0}) .
\]
Let $\Gamma =  \dot{\Gamma}_{C,B}(A_1,A_2)$. By Lemma \ref{l:short-curves} and the above observations,
\begin{equation}
\label{e:gammabig}
\Mod_{d}(\Gamma)\gec 1.
\end{equation}
Note that since the curves in $\Gamma$ start and end in $\frac{1}{2}B$, they must be contained in $CB$ (otherwise their lengths would be at least $2(C-1)r_{B}$, which exceeds $Cr_{B}$ for $C$ chosen large, and thus their length would be too big). Let
\[
\rho(x) =r_{B}^{-1}\exp\ps{- \int_{0}^{r_{B}} \eta(x,r)^{2}\frac{dr}{r}}\one_{CB}.
\]
{\bf Claim 1:} Some multiple (depending on the constants in the lemma) of $\rho$ is admissible for $\Gamma$, that is,
\begin{equation}
\label{e:padmissible}
\int_{\gamma} \rho\gec_{c} 1 \;\; \mbox{ for all } \;\;\gamma\in \Gamma.
\end{equation}

We now show how to finish the lemma assuming the claim: Let $m>0$ and define
\[
E_{B}^{V}=\{x\in CB: \rho\geq mr_{B}^{-1}\}.
\]
Since $\rho\leq r_{B}^{-1}\one_{CB}$, 
\begin{align*}
1& \stackrel{\eqref{e:gammabig}}{ \lec} \Mod_{d}(\Gamma)\stackrel{\eqref{e:padmissible}}{ \lec}  \int \rho^{d}
\leq \int_{E_{B}^{V}}\rho^{d} + \int_{CB\backslash E_{B}^{V}}\rho^{d}
  \lec |E_{B}^{V}|r_{B}^{-d} +  m^{d}r_{B}^{-d}|CB|\\
& \lec  \frac{|E_{B}^{V}|}{|CB|} + m^{d}
\end{align*}
and so for $m>0$ small enough, we have $\frac{|E_{B}^{V}|}{|CB|}\gec 1$, and $E_{B}^{V}$ satisfies the conclusions of the Lemma, which finishes the proof of the Lemma. \\

Now we focus on proving \eqref{e:padmissible} and Claim 1. Let $\gamma\in \Gamma$ (see Figure \ref{f:gamma} for reference). %
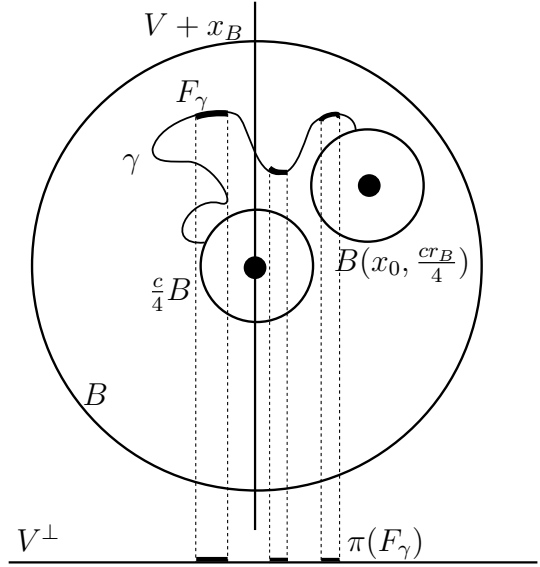
\begin{figure}
\includegraphics[width=200pt]{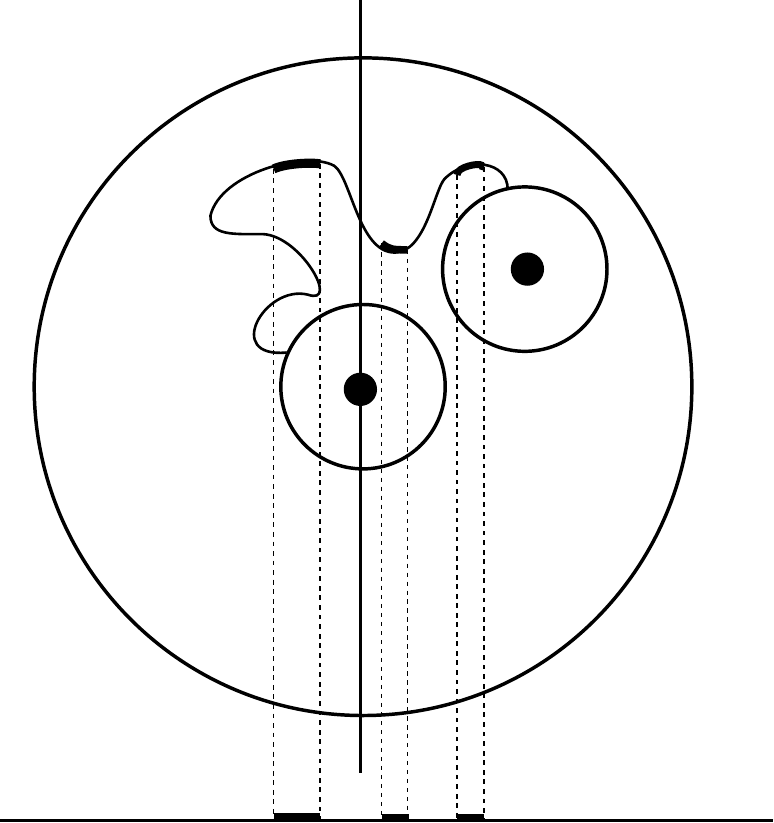}
\begin{picture}(0,0)(200,0)
\put(0,5){$V^{\perp}$}
\put(125,5){$\pi(F_{\gamma})$}
\put(60,175){$F_{\gamma}$}
\put(48,200){$V+x_{B}$}
\put(120,111){$B(x_{0},\frac{cr_{B}}{4})$}
\put(50,100){$\frac{c}{4}B$}
\put(25,60){$B$}
\put(40,150){$\gamma$}
\end{picture}
\caption{The geometric gist of showing that $\int_{\gamma}\rho$ is large is that since $\gamma$ goes between $\frac{c}{4}B$ and $B(x_{0},\frac{cr_{B}}{4})$ and $x_0$ is far from $x_{B}+V$, it must have large projection into $V^{\perp}$, we can then find a large subset $F_{\gamma}$ where the projection is bi-Lipschitz and hence lies in a Lipschitz graph. This implies that the images of the affine maps that best approximate $\gamma$ over this set can't make too steep an angle with $V^{\perp}$, and so must have large angle with $V$. We can use Dorronsoro's theorem to estimate the distance of these affine images to $\gamma$ and hence to $X$, giving us an upper bound on the square integral of $\eta$ on $F_{\gamma}$ and thus a lower bound on $\rho$ on $F_{\gamma}$.}
\label{f:gamma}
\end{figure}
Recall that $\gamma\subseteq CB$ by the definition of $\Gamma$. We just need to show
\begin{equation}
\label{e:rhoQadmissible}
\int_{\gamma} \rho \gec 1.
\end{equation}

Let $\gamma:[0,\ell(\gamma)]\rightarrow \R^{n}$ denote its $1$-Lipschitz arclength parametrization. 

Without loss of generality, we'll scale things so that $[0,\ell(\gamma)]=[0,1]$ (recall that since $\gamma\in \Gamma$, we now have $1=\ell(\gamma)\sim r_{B}$) and translate so that $\gamma(0)=0$. 

We recall the following theorem.

\begin{theorem}[Dorronsoro, \cite{Dor85}, Theorem 6]
Let $1\leq p<p(d)$ where 

\begin{equation}
\label{e:pd}
p(d):= \left\{ \begin{array}{cl} \frac{2d}{d-2} & \mbox{if } d>2 \\
 \infty & \mbox{if } d\leq 2\end{array}\right. . 
 \end{equation}
 For $x\in \bR^{d}$, $r>0$, and $f\in W^{1,2}(\bR^{d})$, define
\[\Omega_{f,p}(x,r)=\inf_{A} \ps{\avint_{B(x,r)} \ps{\frac{|f-A|}{r}}^{p}}^{\frac{1}{p}}\]
where the infimum is over all affine maps $A:\bR^{d}\rightarrow \bR$. Set
\[\Omega_{p}(f):=\int_{\bR^{d}}\int_{0}^{\infty}\Omega_{f,p}(x,r)^{2}\frac{dr}{r} dx. \]
Then
\[
\Omega_{p}(f)\lec_{d,p} ||\grad f||_{2}^{2}.\]
\label{t:dorronsoro}
\end{theorem}
We can extend $\gamma$ to the whole real line by setting $\gamma(t)=\gamma(0)$ for $t\not\in [0,\ell(\gamma)]$ (recall that by our definition of $\Gamma$, $\gamma(0)=\gamma(\ell(\gamma))$, so our extension is $1$-Lipschitz on all of $\R$). For an interval $I$ in the real line, let 
\[
\Omega(I) =  |I|^{-1}\inf_{A} ||\gamma-A||_{L^{\infty}(I)}.
\]
where the infimum is over all affine functions $A:\R\rightarrow \R^{n}$. Then it is not hard to show using Dorronsoro's theorem  and that $\gamma$ is $1$-Lipschitz (see for example \cite[Lemma 2.5]{Azz16a}) that,
\begin{equation}
\label{e:sumIOmegalec1}
\sum_{I\subseteq [0,1]} \Omega(3I)^{2} |I|
\lec 1
\end{equation}
where the sum is over all dyadic intervals in $[0,1]$. Below, all intervals $I$ will refer to a dyadic interval. Let $\pi$ be the orthogonal projection into $V^{\perp}$, $\alpha>0$ (a small number to be chosen later), and $\cI_{1}$ be those intervals $I\subseteq [0,1]$ (not necessarily dyadic) for which 
\begin{equation}
\label{e:pigamma<alpha}
\diam(\pi\circ \gamma(I))<\alpha |I|.
\end{equation}
Let $A>0$ and let $\cI_{2}$ be those maximal dyadic intervals $I\subseteq [0,1]$ for which
\[
\sum_{I\subseteq J\subseteq [0,1]}\Omega(3J)^{2}>A.
\]
{\bf Claim 2:} If 
\[
E_{\gamma} = [0,1]\backslash\bigcup_{I\in \cI_{1}\cup \cI_{2}}I, \;\; F_{\gamma} = \gamma(E_{\gamma}).
\]
then
\begin{equation}
\label{e:egamma>c/4}
|E_{\gamma}|\geq |F_{\gamma}|\geq c 
\end{equation}
and for some constant $c'$
\begin{equation}
\label{e:p>1onEgamma}
\rho\geq c'\one_{F_{\gamma}}.
\end{equation}
These two statements then imply Claim 1 (i.e. \eqref{e:rhoQadmissible}), and thus finish the proof of the Lemma. So now we focus on proving Claim 2.

We first prove \eqref{e:egamma>c/4}. By Chebychev's inequality and \eqref{e:sumIOmegalec1},
\begin{equation}
\label{e:1/A}
\av{\bigcup_{\cI_{2}}I}\lec \frac{1}{A}.
\end{equation}
Let $\{I_{j}\}$ be a subcollection of the intervals $\cI_{1}$ so that no point in $\bigcup_{I\in \cI_{1}} I$ is contained in more than two of the $I_{j}$ and $\bigcup_{I\in \cI_{1}} I=\bigcup_{j}I_j $. Then
\begin{align}
\av{\pi\circ \gamma\ps{\bigcup_{I\in \cI_{1}}I}}
& = \av{\pi\circ \gamma\ps{\bigcup I_{j}}}
\leq \sum |\pi\circ \gamma(I_{j})|
\stackrel{\eqref{e:pigamma<alpha}}{<} \sum_{j} \alpha |I_{j}| \notag \\
& \leq 2\alpha \av{\bigcup I_{j}} \leq 2\alpha.
\label{e:<2delta}
\end{align}
Since $\gamma\in \Gamma$ connects $\frac{c}{4}B$ to $B(x_{0},\frac{cr_{B}}{4})$, and since we're assuming now $r_{B}\sim 1$, we know
\begin{equation}
\label{e:pigamma>c/2}
\pi(\gamma([0,1]))=|\pi(\gamma)|\geq \frac{c}{2}r_{B}\gec c.
\end{equation}
(recall we are using $\gamma$ to denote both the function and its image). This means that for $\alpha$ small enough and $A$ large enough (depending on $c$), and since $\gamma$ and $\pi$ are $1$-Lipschitz,
\begin{align*}
|E_{\gamma}|\geq \av{F_{\gamma}}
& \geq |\pi( F_{\gamma})|
 \stackrel{ \eqref{e:1/A}  \atop\eqref{e:<2delta} }{\geq} |\pi\circ \gamma ([0,1])| - 2\alpha - \frac{C'}{A}
\stackrel{ \eqref{e:pigamma>c/2}}{\gec} c.
\end{align*}
This proves \eqref{e:egamma>c/4}. 

Now we focus on \eqref{e:p>1onEgamma}. It suffices to show the following claim:

\noindent {\bf Claim 3:} for $\epsilon>0$ small enough, if $x\in F_{\gamma}$, $x'\in E_{\gamma}$ is such that $\gamma(x')= x$, $I$ is any dyadic interval containing $x'$ and $r>0$ is such that
\begin{equation}
\label{e:I<r<I}
 |I|/2\leq 8r/\alpha< |I|
 \end{equation}
 then 
 \begin{equation}
 \label{e:eta<omega}
 \eta(x,r)\lec_{\alpha} \Omega(3I).
 \end{equation}
 Indeed, if this is the case, then for each $x\in F_{\gamma}$
\begin{align*}
\int_{0}^{r_{B}} \eta(x,r)^{2}\frac{dr}{r}
& \leq \int_{\alpha/8}^{1}\frac{dr}{r}+\sum_{x'\in I\subseteq [0,1]} \int_{|I|\alpha/16}^{|I|\delta/8} \eta(x,r)^{2}\frac{dr}{r}\\
& \lec_{\alpha} 1+\sum_{x'\in I\subseteq [0,1]} \int_{|I|\alpha/16}^{|I|\alpha/8} \Omega(3I)^{2}  \frac{dr}{r} \\
& \lec 1+ \sum_{x'\in I\subseteq [0,1]} \Omega(3I)^{2}
\lec \frac{A}{\ve^{2}}
\end{align*}
where in the last line we used the fact that $x'\in E_{\gamma}$ implies no dyadic interval $I\subseteq [0,1]$ containing $x'$ is in $\cI_2$. This  then implies $\rho(x)\gec 1$ for each $x\in F_{\gamma}$, proving \eqref{e:p>1onEgamma} and thus Claim 2. So now we focus on Claim 3.

Let $x,x', I,r$ be as in the claim. Let $\ve>0$, which will be fixed later and depend on $\alpha$. We can assume $\Omega(3I)<\ve$, for otherwise 
\[
\eta(x,r)\leq  1\leq \ve^{-1}\Omega(3I)\lec_{\alpha} \Omega(3I).
\]
To prov \eqref{e:eta<omega}, we need to produce a line $L_{x,r}$ passing through $x$ whose distance to $\gamma$ (and hence to $X$) is no more than a constant times $\Omega(3I)$, and also whose angle with $V$ is at least $\theta$ (where $\theta$ we will choose below and will depend on $\alpha$).

Let $A_{I}:\R\rightarrow \R^{n}$ be the affine map that achieves the infimum in $\Omega(3I)$, and let $L_{I} = A_{I}(\R)$. Since $x'\in E_{\gamma}$, we know $I$ is not contained in an interval from $\cI_{1}$, and so $\diam(\pi \circ\gamma(I))\geq \alpha|I|$. Since $\Omega(3I)<\ve$, for $\ve \ll \alpha$ we then have that 
\begin{equation}
\label{e:delta/2I}
\diam A_{I}(I) \geq  \diam \pi\circ A_{I}(I) \geq \diam \pi\circ \gamma(I) - 2\Omega(3I)|3I|
>\frac{\alpha}{2}|I|
\end{equation}
which implies
\begin{equation}
\label{e:A-lower-lip}
|A_{I}(a)-A_{I}(b)|\geq \frac{\alpha}{2} |a-b| \mbox{ for all }a,b\in \R
\end{equation}
Then
\begin{equation}
\label{e:x-A_I(x')}
|x-A_{I}(x')|
=|\gamma(x')-A_{I}(x')|
\leq  \Omega(3I)|3I| < \ve |3I|.
\end{equation}
Thus, we can choose  $\ve\ll 1$ so that 
\begin{align*}
A_{I}^{-1}(B(x,2r))
& \subseteq A_{I}^{-1}(B(A_I(x'),2r+|x-A_I(x')|))\\
& \stackrel{\eqref{e:x-A_I(x')}\atop \eqref{e:A-lower-lip}}{\subseteq} B(x',2(2r+\ve |3I| )/\alpha)
\stackrel{\eqref{e:I<r<I}}{\subseteq} B(x',8r/\alpha)\stackrel{\eqref{e:I<r<I}}{\subseteq} 3I
\end{align*}

Then
\begin{align*}
\sup_{z\in L_{I}\cap B(x,2r)}\dist(z,X) & \leq  \sup_{z\in L_{I}\cap B(x,2r)}\dist(z,\gamma)
 \leq \sup_{y\in A_{I}^{-1}(B(x,2r))}|A_{I}(y)-\gamma(y)|\\
& \leq \sup_{y\in 3I } |A_{I}(y)-\gamma(y)|
=|3I| \Omega(3I)
\stackrel{\eqref{e:I<r<I}}{\leq} \frac{48}{\alpha}r \Omega(3I)
\end{align*}
Now, the line $L_{I}$ doesn't pass through $x$ necessarily, so let $L_{x,r}$ be the line parallel to $L_{I}$ passing through $x$. Then for $\ve<\frac{\alpha}{48}$,
\[
\dist(L_{x,r},L_{I}) 
\stackrel{\eqref{e:x-A_I(x')}}{\leq} |3I|\Omega(3I) 
\stackrel{\eqref{e:I<r<I}}{\leq} \frac{48}{\alpha}  r\ve <r
\]
and so for each $z\in L_{x,r}\cap B(x,r)$, the closest point to $z$ in $L_{I}$ is contained in $B(x,2r)$. Thus,
\begin{equation}
\label{e:dist<96}
\sup_{z\in L_{x,r}\cap B(x,r)}\dist(z,X)
\leq |3I|\Omega(3I)+\sup_{z\in L_{I}\cap B(x,2r)}\dist(z,X) 
\leq  \frac{96}{\alpha}r \Omega(3I).
\end{equation}
Finally, by \eqref{e:delta/2I},
\[
\angle(L_{x,r},V)= \angle(L_{I},V) =\av{\pi|_{L_{I}}}\gec \alpha.
\]
and so if we pick $\theta\ll\alpha$, then $\angle(L_{I},V) \geq \theta$. This and \eqref{e:dist<96} imply \eqref{e:eta<omega}, finishing the proof of Claim 3 and thus the lemma.

\end{proof}

\subsection{Proof of Lemma \ref{l:eta-carleson}: Part II}

To finish the proof of Lemma \ref{l:eta-carleson}, we now need to improve the local estimate in the previous lemma to a Carleson estimate. That is, we know inside each ball $B$ there is a large set where $\int_{0}^{r_{B}}\eta_{X}^{V,\theta}(x,r)^2\frac{dr}{r}$ is $L^\infty$, but we would like to improve this to knowing that the average of this integral in $B$ (or in a cube, as we will really show) is bounded. 

First notice that since $X$ is Ahlfors $d$-regular, for any $V\in \cG(n,d-1)$ and any cube $Q\in \cD$, the ball $B_{Q}':=\frac{c_{0}}{C}B_{Q}$ satisfies the conditions of Lemma \ref{l:generalV} (recall $c_0$ from the definion of Christ-David cubes) for some $c$ depending on the Ahlfors regularity constants (in other words, an Ahlfors $d$-regular set $X$ cannot be locally concentrated around a $(d-1)$-dimensional plane). Let
\[
E_{Q} = E_{B_{Q}'}^{V}\subseteq c_{0}B_{Q}\cap X\subseteq Q. 
\]

We will use the following lemma. 

\begin{lemma} \cite[Lemma IV.1.12]{of-and-on} Let $\alpha:\cD\rightarrow [0,\infty)$ be given and suppose there is $N>0$ and $\eta>0$ so that for all $R\in \cD$,
\[
\av{\ck{x\in R : \sum_{x\in Q\subseteq R}\alpha(Q)\leq N}}\geq \eta |R|.
\]
Then 
\[
\sum_{Q\subseteq R}\alpha(Q)\lec_{N,\eta} |R| \mbox{ for all }R\in \cD.
\]
\end{lemma}

Thus, our lemma will follow once we show the following claim. 

\noindent {\bf Claim 1:} if for $\delta>0$ and $Q\in \cD$ we set
\[
\eta_{\delta} (Q)=\isif{ 1 & \eta_{X}^{V,\theta}(Q)  \geq \delta \\ 
0 &  \eta_{X}^{V,\theta}(Q) < \delta}
\]
then for each $R\in \cD$ there is $G_{R}\subseteq R$ so that $|G_{R}|\gec |R|$ and
\[
\sum_{x\in Q\subseteq R}\eta_{\delta}(Q)\lec_{\delta} 1 \;\; \mbox{ for all } \;\; x\in G_{R}.
\]

Recall that $|E_{R}|\geq c|R|$ for some constant $c$ by Lemma \ref{l:generalV} and the definition of $E_{R}$. Let $Q_{j}$ be the maximal cubes in $R$ for which $|Q_{j}\cap E_{R}|<\frac{c}{2} |Q_{j}|$. Then
\[
\sum |Q_{j}\cap E_{R}|
< \frac{c}{2}\sum |Q_{j}|\leq \frac{c}{2}|R|\leq \frac{1}{2}|E_{R}|
\]
so if we set $F_{R}= E_{R}\backslash \bigcup Q_{j}$, then 
\[
|F_{R}|\geq \frac{1}{2} |E_{R}|\geq \frac{c}{2} |R|
\]
and $F_{R}$ has the property that
\begin{equation}
\label{e:Qcap F_R}
|E_{R}\cap Q|\geq \frac{c}{2} |Q| \;\;\mbox{ whenever }\;\; Q\cap F_{R}\neq\emptyset \mbox{ and }Q\subseteq R.
\end{equation}

We require the following lemma.

\begin{lemma} \cite[Appendix]{AM16}
Let $\mu$ be a $C_{\mu}$-doubling measure and let $\cD$ the cubes from Theorem \ref{t:Christ} for $X=\supp \mu$ with admissible constants $c_{0}$ and $\rho$. Let $E\subseteq Q_{0}\in \cD$, $M>1$, $\delta>0$, and set
\begin{multline*}
\cP=\{Q\subseteq Q_0: 
Q\cap E\neq\emptyset, \exists \; \xi\in B(\zeta_{Q},M\ell(Q)) \\
\mbox{ such that } \dist(\xi,E)\geq \delta\ell(Q)\}.
\end{multline*}
Then there is $C_{1}=C_{1}(M,\delta,C_{\mu})>0$ so that, for all $Q'\subseteq Q_{0}$,
\begin{equation}
\sum_{Q\subseteq Q' \atop Q\in \cP} \mu(Q)\leq C_{1} \mu(Q').
\label{e:sumP}
\end{equation}
\label{l:porous}
\end{lemma}

Let $\cP_{R}$ be the cubes from Lemma \ref{l:porous} with $M=2$, $\mu=\cH^{d}|_{X}$, $E=F_{R}$, and $\delta/4$ in place of $\delta$, and set 
\[
\cC_{R}=\{Q\subseteq R: Q\cap F_{R}\neq\emptyset \;\; \mbox{ and }\;\; Q\not\in \cP_{R}\}.\]  

\noindent {\bf Claim 2:} for $\phi>0$ small enough depending on $\delta$, if $Q\in \cC_{R}$, $\eta=\eta_{X}^{V,\theta}$, and 
\[
\tilde{Q} := \bigcup\{T\in \cD:  \ell(T)=\ell(Q) \mbox{ and }T\cap  3B_{Q}\neq\emptyset\}
\]
then $\eta(Q)<\delta$ if
\begin{equation}
\label{e:E_Rcap Q-int}
\avint_{E_{R}\cap \tilde{Q} }\int_{\ell(Q)}^{\ell(Q)/\rho} \eta(x,r)^{2}\frac{dr}{r} <\phi 
\end{equation}
where $\rho$ is as in the definition of Christ-David cubes.

Before proving this, let's show how it implies Claim 1: note that the sets $\{\tilde{Q}: Q\in \cD_{k}\}$ have bounded overlap for each $k$, and if $\eta(Q)\geq \delta$, then the reverse inequality in \eqref{e:E_Rcap Q-int} holds, and so
\begin{align*}
\sum_{Q\in \cC_{R}}\eta_{\delta}(Q)|Q|
& \leq \sum_{Q\in \cC_{R} \atop \eta(Q)\geq \delta }|Q|
 \stackrel{\eqref{e:Qcap F_R}}{\leq} \frac{2}{c}  \sum_{Q\in \cC_{R} \atop \eta(Q)\geq \delta } |E_{R}\cap Q|\\
& \stackrel{\eqref{e:E_Rcap Q-int}}{\leq}  \frac{2}{c\phi } \sum_{Q\subseteq R\atop Q\cap F_{R}\neq\emptyset  } \int_{E_{R}\cap \tilde{Q}}\int_{ \ell(Q)}^{\ell(Q)/\rho} \eta(x,r)^{2}\frac{dr}{r} \\
& \lec \frac{2}{c\phi }  \int_{E_{R}}\int_{0}^{\ell(R)/\rho} \eta(x,r)^{2}\frac{dr}{r}
\stackrel{\eqref{e:etaintonE}}{\lec}_{\phi} |R|.
\end{align*}

Thus,
\begin{equation}
\label{e:Qcap FR sum}
\sum_{Q\subseteq R\atop Q\cap F_{R}\neq\emptyset} \eta_{\delta}(Q)|Q|
\leq \sum_{Q\in \cP_{R}}|Q|+ \sum_{Q\in \cC_{R}}\eta_{\delta}(Q)|Q|
\stackrel{\eqref{e:sumP}}{\lec}_{\phi} |R|.
\end{equation}

Let $N$ be a large integer and 
\[
G_{R} =\ck{x\in F_{R}: \sum_{x\in Q\subseteq R}  \eta_{\delta}(Q) \leq N}.
\]
Then
\begin{align*}
|F_{R}\backslash G_{R}|
& \leq \frac{1}{N}\int_{F_{R}\backslash G_{R}} \sum_{x\in Q\subseteq R}  \eta_{\delta}(Q)\\
& = \frac{1}{N}\sum_{Q\subseteq R\atop Q\cap F_{R}\neq\emptyset} \eta_{\delta}(Q) |(F_{R}\backslash G_R)\cap Q|\\
& \leq  \frac{1}{N}\sum_{Q\subseteq R\atop Q\cap F_{R}\neq\emptyset} \eta_{\delta}(Q) | Q|
\stackrel{\eqref{e:Qcap FR sum}}{\lec}_{\phi} \frac{|R|}{N}.
\end{align*}

Thus, for $N$ large enough (depending on $\phi$)
\[
|G_{R}|\geq \frac{|F_{R}|}{2} \geq \frac{c}{4}|R|
\]
which proves Claim 1. 

Thus, it remains to show Claim 2, which we now focus on. By \eqref{e:eta-monotone}, \eqref{e:E_Rcap Q-int} implies
\begin{equation}
\label{e:eta-over-eR<phi}
\avint_{E_{R}\cap \tilde{Q}}\eta(x, \ell(Q))^{2} \lec \phi.
\end{equation}
By Chebychev, this implies that if 
\[
S_{Q}:=\{x\in E_{R}\cap \tilde{Q}:\eta(x, \ell(Q)) < \phi^{\frac{1}{4}}\}
\]
then
\begin{equation}
\label{e:ERS_Q}
|\tilde{Q}\cap E_{R} \backslash S_{Q}|
\leq \phi^{-1/2} \int_{E_{R}\cap \tilde{Q}} \eta(x, \ell(Q))^{2} 
\stackrel{\eqref{e:eta-over-eR<phi}}{\lec} \phi^{\frac{1}{2}}|E_{R}\cap \tilde{Q}|\lec \phi^{\frac{1}{2}}|Q|
\end{equation}

We will show that 
\begin{equation}
\label{e:distxSQ}
\dist(x,S_{Q})<\frac{\delta}{2} \ell(Q) \;\; \mbox{ for all }x\in Q, \;\; Q\in \cC_{R}.
\end{equation}

If we prove this, then by \eqref{e:eta-lip}, for $\phi^{\frac{1}{4}}<\frac{\delta}{2}$, 
\[
\eta(x, \ell(Q)) < \phi^{\frac{1}{4}}+\frac{ \frac{\delta}{2} \ell(Q)}{\ell(Q)} <\delta \;\;\mbox{ for all $x\in Q$}
\]
which proves Claim 2. So now we focus on \eqref{e:distxSQ}. 

Let $x\in Q\in \cC_{R}$. Note that by the definition of $\cC_{R}$ and since $Q\subseteq 2B_Q$, this implies that there is $x'\in F_{R}$ with $|x-x'|<\frac{\delta}{4}\ell(Q)$. If $\dist(x',S_{Q})=0$, then
\[
\dist(x,S_{Q})=|x-x'|<\frac{\delta}{4} \ell(Q)\] 
and we're done, so assume $\dist(x',S_{Q})>0$. For $\delta>0$ small, $x'\in 2B_{Q}\cap X\subseteq \tilde{Q}$. Since $x'\in F_{R}\subseteq R$, we can pick a maximal cube $Q'\subseteq R\cap \tilde{Q}\backslash S_Q$ containing $x'$. Since $Q'\cap F_{R}\neq\emptyset$, \eqref{e:Qcap F_R} implies 
\begin{equation}
\label{e:Q'<phi}
|Q'|
\stackrel{\eqref{e:Qcap F_R}}{\leq} \frac{2}{c} |E_{R}\cap Q'|
\leq \frac{2}{c} |E_{R}\cap \tilde{Q} \backslash S_{Q}|
\stackrel{\eqref{e:ERS_Q}}{\lec} \phi^{\frac{1}{2}}|Q|.
\end{equation}
By \eqref{e:ERS_Q}, for $\phi$ small, $S_{Q}\cap T\neq\emptyset$ for each cube $T$ with $\ell(T)=\ell(Q)$ intersecting $3B_{Q}$. Hence, we must have $\ell(Q')<\ell(Q)$. Thus, since $Q'$ is maximal, its parent $Q''$ must intersect $S_Q$, and so by Ahlfors regularity,
\[
\dist(x',S_{Q})\leq \diam Q''\lec  \ell( Q') \stackrel{\eqref{e:Q'<phi}}{\lec} \phi^{\frac{1}{2d}}\ell(Q).
\]
Thus, if we choose $\phi$ small enough so that $\dist(x',S_{Q})<\frac{\delta}{4}\ell(Q)$, we then have 
\[
\dist(x,S_{Q})\leq |x-x'|+\dist(x',S_{Q})<\frac{\delta}{4}\ell(Q)+\frac{\delta}{4}\ell(Q)=\frac{\delta}{2}\ell(Q)
\]
which proves \eqref{e:distxSQ} and thus finishes the proof of Lemma \ref{l:eta-carleson}.

\section{The proof of the Weak Geometric Lemma}

We now use very similar arguments to prove Lemma \ref{l:wgl}. 

\subsection{$\xi$-numbers}

We introduce a new quantity similar to the $\eta$-numbers that detects whether, given a $d$-plane $V$, if $X$ is not close to the plane in some ball, there is a line segment through the center of a ball lying close to $X$ and transversal to $V$. 

 For $V\in \cA(n,d)$,  $\beta>0$ and $\theta>0$, we set
\[
\xi_{X,V}^{\theta,\beta}(x,r) = \isif{  0 & \beta_{X}(x,r,V)<\beta \\ \eta_{X}^{V,\theta}(x,r) & \beta_{X}(x,r,V)\geq\beta }.
\]

\begin{lemma}
For $x\in X$, $\theta,\beta>0$, and $V\in \cA(n,d)$,
\begin{equation}
\label{e:xi-monotone}
 \xi_{X,V}^{\theta,\beta s/r}(x,r) \leq \frac{s}{r}  \xi_{X,V}^{\theta,\beta}(x,s) \mbox{ for }r\leq s
 \end{equation}
 and for $\beta'\geq \beta$, 
 \begin{equation}
 \label{e:betabeta'}
 \xi_{X,V}^{\theta,\beta'}(x,s)\leq  \xi_{X,V}^{\theta,\beta}(x,s)
 \end{equation}
 \end{lemma}
 
 \begin{proof}
If $\beta_{X}(x,s,V)<\beta$, then \eqref{e:betamonotone} implies $\beta_{X}(x,r,V)< \beta s/r$, and so 
\[
 \xi_{X,V}^{\theta,\beta s/r}(x,r) =0\leq  \frac{s}{r}  \xi_{X,V}^{\theta,\beta}(x,s).
 \]
 Otherwise, if  $\beta_{X}(x,s,V)\geq \beta$, then 
\[
\frac{s}{r}  \xi_{X,V}^{\theta,\beta}(x,s)
=\frac{s}{r}  \eta_{X}^{V,\theta}(x,s)
\stackrel{\eqref{e:eta-monotone}}{\geq} \eta_{X}^{V,\theta}(x,r)
\geq  \xi_{X,V}^{\theta,\beta s/r}(x,r) .
\]
This proves \eqref{e:xi-monotone}. For \eqref{e:betabeta'}, if $\beta_{X}(x,s,V)< \beta'$, then $ \xi_{X,V}^{\theta,\beta'}(x,s)=0$ and \eqref{e:betabeta'} follows; if $\beta_{X}(x,s,V)\geq \beta'\geq \beta$, then 
\[
\xi_{X,V}^{\theta,\beta'}(x,s)=\eta_{X}^{V,\theta}(x,s)=\xi_{X,V}^{\theta,\beta}(x,s).
\]
 \end{proof}

\subsection{Finding balls that are close to a plane or have segments transversal to the plane}

\begin{lemma}
\label{l:xi-int}
There is $\beta_0>0$ (depending on $n$, the Loewner constants, and the Ahlfors regularity constants), so that for all $0<\beta<\beta_0$, there is $\theta>0$ depending on $\beta$ so that the following holds. If $V\in \cG(n,d)$, then for all $Q\in \cD$ there is $E_{Q,V}'\subseteq Q$ with $|E_{Q}'|\gec_{\beta} |Q|$ and 
\begin{equation}
\label{e:int eta<1}
\int_{0}^{\ell(Q)}\xi_{X,V}^{\theta,\beta}(x,r)^{2}\frac{dr}{r} \lec_{\beta} 1 \;\; \mbox{ for all } x\in E_{Q,V}'.
\end{equation}
\end{lemma}

\begin{proof}
Let $\xi=\xi_{X,V}^{\theta,\beta}$, where $\theta$ and $\beta$ will be picked later. For $Q\in \cD$, let
\[
\beta_{X}(Q,V) =  \sup_{y\in 3B_Q\cap X} \frac{\dist(y,V)}{\ell(Q)}, \;\; \beta_{X}(Q) = \inf_{V\in \cA(n,d)} \beta_{X}(Q,V).
\]
Let $Q\in \cD$, let $\check{Q}$ be the maximal cube in $Q$ containing the same center as $Q$ so that $\ell(\check{Q})<\frac{c_0\ell(Q)}{3C}$ (where $C$ is as in Lemma \ref{l:generalV}), let $Q_{j}$ be those maximal cubes in $\check{Q}$ for which $\beta_{X}(Q_{j},V)\geq \beta^2$, and set $S=\bigcup Q_{j}$. Note that there could be no such cubes in which case $S=\emptyset$. \\

\noindent {\bf Case 1:} If $|S|<\frac{1}{2}|\check{Q}|$, then set $E_{Q,V}'=Q\backslash S$, so then 
\[
|E_{Q,V}'|
\geq \frac{1}{2}|\check{Q}|\gec |Q|.
\]
For $x\in E_{Q,V}'$ and $0<r<\ell(Q)$, we can find $Q'\subseteq Q$ containing $x$ so that $\rho \ell(Q')< r\leq \ell(Q')$. Since $x\in Q'\cap  E_{Q,V}'$ and $B(x,r)\subseteq B(x,\ell(Q))\subseteq 3B_Q$, we have
\[
\beta_{X}(x,r,V)
\stackrel{\eqref{e:betamonotone}}{ \lec} \beta_{X}(x,\ell(Q'),V)
\leq \beta_{X}(Q',V)<\beta^{2} .
\]
Thus, for $\beta>0$ small enough, $\beta_{X}(x,r,V)<\beta$ and so $\xi(x,r)=0$. In particular, since $0<r<\ell(Q)$ was arbitrary, we have
\[
\int_{0}^{\ell(Q)}\xi(x,r)^{2}\frac{dr}{r}=0 \;\; \mbox{ for all }x\in E_{Q,V}'.
\]

\noindent {\bf Case 2:} Alternatively, suppose $|S|\geq \frac{|\check{Q}|}{2}$. Since $\beta_{X}(3B_{Q_j},V)\geq \beta$, Lemma \ref{l:generalV} implies there is a set 
\[
E_{3B_{Q_j}}^{V}\subseteq 3CB_{Q_j}\cap X\subseteq Q\] 
upon which 
\[
\int_{0}^{3\ell(Q_{j})} \eta_{X}^{V,\theta}(x,r)^2\frac{dr}{r}\lec_{\beta} 1 \mbox{ for all }x\in E_{3B_{Q_j}}^{V}
\]
for some $\theta$ small enough depending on $\beta$. Let $j_{k}$ be such that $\{3CB_{Q_{j_{k}}}\}$ is a $5r$-subcovering of $\{3CB_{Q_{j}}\}$. Since the balls $\{3CB_{Q_{j_{k}}}\}$ are disjoint, if we set 
\[
E_{Q,V}' = \bigcup E_{3B_{Q_j}}^{V}
\subseteq \bigcup 3CB_{Q_{j_{k}}}\cap X
\subseteq 3CB_{\check{Q}}\cap X\subseteq c_0 B_{Q}\cap X\subseteq Q,
\]
then
\begin{align*}
|E_{Q,V}'|
& =\sum |E_{3B_{Q_{j_{k}}}}|
\gec \sum|Q_{j_{k}}|
\sim \sum |15B_{Q_{j_{k}}}\cap X|
\geq \av{\bigcup 15B_{Q_{j_{k}}}\cap X}\\
& 
\geq \av{\bigcup 3CB_{Q_{j}}} \geq \av{\bigcup Q_{j}} \geq \frac{1}{2} |\check{Q}|\gec |Q|.
\end{align*}

Now let $x\in E_{Q,V}'$ and $0<r<\ell(Q)$. Observe that if  there is a cube $x\in R\subseteq Q$ with $\rho \ell(R)< r\leq \ell(R)$ and $\beta_{X}(R,V)<\beta^2$, then since $B(x,r)\subseteq 2B_{R}$,
\[
\beta_{X}(x,r,V) 
\lec \beta_{X}(Q',V)<\beta^{2}
\]
and so for $\beta$ small enough $\beta_{X}(x,r,V)<\beta$, hence $\xi(x,r)=0$ for such $r$. We will use this observation in the following cases.\\

\noindent {\bf Case 2.1:} Suppose $x\not\in S$, then for every cube $R\subseteq \check{Q}$ containing $x$, $\beta_{X}(R,V)<\beta^2$, and the above observation implies $\eta(x,r)=0$ for all $r\leq \ell(\check{Q})$. Since $\eta(x,r)\leq 1$ for $\ell(\check{Q})\leq r\leq \ell(Q)$, this implies \eqref{e:int eta<1}. \\

\noindent {\bf Case 2.2:} If $x\in S$, then there is $Q_{j}$ containing $x$ and so $x\in E_{3Q_{j}}^{V}\cap Q_j$. \\

\noindent {\bf Case 2.2.1:} If $Q_j\subset\neq Q$, then for $r> \ell(Q_{j})$, there is $Q'\subseteq Q$ properly containing $Q_{j}$ so that $\rho \ell(Q')< r\leq \ell(Q')$. Since $Q_j$ was maximal, $\beta_{X}(Q',V)<\beta^{2}$, and the earlier observation implies $\xi(x,r)=0$. Hence, $\xi(x,r)=0$ for all $r$ such that $\ell(Q_{j})< r\leq \ell(Q)$. Thus, recalling that $x\in E_{3B_{Q_{j}}}^{V}$,
\[
\int_{0}^{\ell(Q)}\xi(x,r)^{2}\frac{dr}{r} 
=\int_{0}^{\ell(Q_j)}\xi(x,r)^{2}\frac{dr}{r} 
\leq \int_{0}^{\ell(Q_j)}\eta_{X}^{V,\theta}(x,r)^{2}\frac{dr}{r} 
\lec_{\beta,X} 1.
\]
\noindent {\bf Case 2.2.2:} If $Q_j=Q$, then the above equation holds yet again.\\

This finishes the proof of the lemma.

\end{proof}

\begin{lemma}
\label{l:xi-carleson}
Choose $\beta$ and $\theta$ so that the conclusions of the previous lemma hold for $\beta\rho/4$ in place of $\beta$. For $Q\in \cD$ and $V\in \cG(n,d)$, let 
\[
\xi_{X,V}^{\theta,\beta}(Q) = \sup_{x\in Q} \xi_{X,V}^{\theta,\beta}(x,2\ell(Q)).\]
Then for all $\delta>0$ and $R\in \cD$,
\begin{equation}
\label{e:xi-carleson}
\sum_{Q\subseteq R \atop \xi_{X,V}^{\theta,\beta}(Q)\geq \delta} |Q|\lec_{\delta,\beta} |R|.
\end{equation}
\end{lemma}

\begin{proof}
This is shown in much the same way as Lemma \ref{l:eta-carleson}, and we will refer to notation and parts of the proof from there. Without loss of generality, we can assume $\delta\in (0,\min\{\beta,1\})$. Let $E_{R} = E_{R,V}'$ be from the previous lemma (but with $\beta\rho/2$ instead of $\beta$),  so again $|E_{R}|\geq c |R|$ for some constant $c$ depending on $\beta$. Let $Q_{j}$ be the maximal cubes in $R$ for which $|Q_{j}\cap E_{R}|<\frac{c}{2} |Q_{j}|$ and set  $F_{R}= E_{R}\backslash \bigcup Q_{j}$, so again we have  $|F_{R}|\geq \frac{c}{2} |R|$ and 
\begin{equation}
\label{e:Qcap F_R2}
|E_{R}\cap Q|\geq \frac{c}{2} |Q| \;\;\mbox{ whenever }\;\; Q\cap F_{R}\neq\emptyset \mbox{ and }Q\subseteq R.
\end{equation}

Define $F_{R}$ just as before, let $\cP_{R}$ be the cubes from Lemma \ref{l:porous}. Define $\cC_{R}$ and $\tilde{Q}$ be as before. Let $\phi>0$ and $Q\in \cC_{R}$ be so that 
\begin{equation}
\label{e:E_Rcap Q-int2}
\avint_{E_{R}\cap \tilde{Q}}\int_{\ell(Q)}^{6\ell(Q)/\rho} \xi_{X,V}^{\theta,\beta\rho/4}(x,r)^{2}\frac{dr}{r} <\phi.
\end{equation}

Using \eqref{e:xi-monotone} and \eqref{e:betabeta'}, we see that for $s\in[3\ell(Q)/\rho,6\ell(Q)/\rho]$,
\begin{align*}
\xi_{X,V}^{\theta,\beta/2}(x,3\ell(Q)) 
& \stackrel{\eqref{e:xi-monotone}}{\leq}  \frac{s}{3\ell(Q)} \xi_{X,V}^{\theta,\beta\frac{3\ell(Q)}{2s}}(x,s)\\
& \stackrel{\eqref{e:betabeta'}}{\leq}
 \frac{2}{\rho}\xi_{X,V}^{\theta,\beta\rho/4}(x,s).
\end{align*}

This and \eqref{e:E_Rcap Q-int2} imply 

\begin{equation}
\label{e:inteta<phi}
\avint_{E_{R}\cap \tilde{Q}}\xi_{X,V}^{\theta,\frac{\beta}{2}}(x, 3\ell(Q))^{2} \lec \phi.
\end{equation}

\noindent {\bf Claim:} for $\phi$ small enough the above inequality implies
\begin{equation}
\label{e:xi<delta}
\xi_{X,V}^{\theta,\beta}(Q)<\delta .
\end{equation}
Assume \eqref{e:inteta<phi}. Again, set
\[
S_{Q} = \{x\in E_{R}\cap \tilde{Q}: \xi_{X,V}^{\theta,\frac{\beta}{2}}(x,3\ell(Q))<\phi^{1/4}\}.\]
Hence, for $\phi$ small enough, we again have by Chebychev
\[
|\tilde{Q}\cap E_{R} \backslash S_{Q}|
<\phi^{\frac{1}{2}}|Q|
\]
and with the same proof as before, for $\phi>0$ small enough (depending on $\delta$ and $\beta$),
\[
\dist(x,S_{Q})< \frac{\delta}{2} \ell(Q) \;\; \mbox{ for all }x\in Q, \;\; Q\in \cC_{R}.\]
Hence, for $x\in Q\in \cC_{R}$, there is $x'\in S_{Q}$ with 
\begin{equation}
\label{e:x-x'<delta/2}
|x-x'|<\frac{ \delta}{2} \ell(Q).
\end{equation}

{\bf Case 1:} If $\beta_{X}(x',3\ell(Q),V)<\frac{\beta}{2}$, then \eqref{e:x-x'<delta/2} and $\delta<1$ imply $B(x,2\ell(Q))\subseteq 3B_{Q})$ and so
\[
\beta_{X}(x,2\ell(Q),V) \leq \frac{3}{2}\beta_{X}(x',3\ell(Q),V)
<\frac{3\beta}{4}<\beta.
\]
and so $\xi_{X,V}^{\theta,\beta}(x,2\ell(Q),V)=0<\delta$. \\

{\bf Case 2:} Alternatively, if $\beta_{X}(x',3\ell(Q),V)\geq\frac{\beta}{2}$, then $\xi_{X,V}^{\theta,\frac{\beta}{2}}(x',3\ell(Q))=\eta_{X}^{V,\theta}(x',3\ell(Q))$, and so for $\phi>0$ small enough (depending on $\delta$), because $x'\in S_{Q}$,
\begin{align*}
\xi_{X,V}^{\theta,\beta} (x,2\ell(Q),V) & \leq \eta_{X}^{V,\theta}(x,2\ell(Q),V) \\
& \stackrel{\eqref{e:eta-lip}}{\leq} 
\frac{|x-x'|}{2\ell(Q)}+
\eta_{X}^{V,\theta}(x',2\ell(Q),V)\\
&   \stackrel{\eqref{e:eta-monotone} \atop \eqref{e:x-x'<delta/2}}{\leq}  \frac{\delta}{4}  + \frac{3}{2}\eta_{X}^{V,\theta}(x',3\ell(Q),V)\\
&  =  \frac{\delta}{4} +\frac{3}{2}\xi_{V}^{\theta,\frac{\beta}{2}}(x',3\ell(Q))< \frac{\delta}{4} +\frac{3}{2}\phi^{\frac{1}{4}}<\delta.
\end{align*}
This proves \eqref{e:xi<delta} and hence the claim.  In particular, 
\begin{equation}
\label{e:E_Rcap Q-int2-implication}
\mbox{ if } \;\; \xi_{X,V}^{\theta,\beta}(Q)\geq \delta, \;\; \mbox{ then }\;\; \avint_{E_{R}\cap Q}\int_{\ell(Q)}^{4\ell(Q)/\rho} \xi_{X,V}^{\theta,\beta}(x,r)^{2}\frac{dr}{r} \geq \ve.
\end{equation}

Let
\[
\xi_{\delta} (Q)=\isif{ 1 & \xi_V(Q)  \geq \delta \\ 
0 &  \xi_V(Q) < \delta}
\]
Then by our choice of $E_{R}$,
\begin{align*}
\sum_{Q\in \cC_{R}}\xi_{\delta}(Q)|Q|
& \leq \sum_{Q\in \cC_{R} \atop \xi_{X,V}^{\theta,\beta}(Q)\geq \delta }|Q|
 \leq \frac{2}{c}  \sum_{Q\in \cC_{R} \atop \xi_{X,V}^{\theta,\beta}(Q)\geq \delta } |E_{R}\cap Q|\\
& \stackrel{\eqref{e:E_Rcap Q-int2-implication}}{\leq}  \frac{2}{c\ve } \sum_{Q\subseteq R\atop Q\cap F_{R}\neq\emptyset  } \int_{E_{R}\cap Q}\int_{ \ell(Q)}^{6\ell(Q)/\rho} \xi_{X,V}^{\theta,\beta\rho/4}(x,r)^{2}\frac{dr}{r} \\
& \lec \frac{2}{c\ve }  \int_{E_{R}}\int_{0}^{6\ell(R)/\rho} \xi_{X,V}^{\theta,\beta\rho/4} (x,r)\frac{dr}{r}
\lec_{\phi} |R|.
\end{align*}

Just as before, for each $R\in \cD$ we can now find $G_{R}\subseteq R$ so that $|G_{R}|\gec |R|$ and
\[
\sum_{x\in Q\subseteq R}\xi_{\delta}(Q)\lec_{\delta} 1 \;\; \mbox{ for all } \;\; x\in G_{R}.
\]

This completes the proof.

\end{proof}

\subsection{$\eta$ and $\xi$ small imply $\beta$ is small}

\begin{lemma}
\label{l:eta-xi-beta}
For all $0<\beta<\beta_{0}\rho/4$, there are $\theta,\ve_0>0$ so that the following holds: for all $\ve\in (0,\ve_{0})$, if $\cV$ is a maximally $\ve$-separated set in $\cG(n,d)$, and for  $Q\in \cD$ we define 
\[
\xi_{X}^{\theta,\beta}(Q) = \sum_{V\in \cV} \xi_{X,V}^{\theta,\beta}(Q).
\]
and if $\eta_X^{\theta}(Q)<\ve$ and $\xi_{X}^{\theta,\beta}(Q)<\ve$, then $\beta_{X}(B_{Q})\lec \beta$.
\end{lemma}

\begin{proof}
Let $\theta'>0$, which will be determined shortly and will depend on $\theta$ (and so ultimately on $X$ and $\beta$, but not on $
\ve$). Suppose $\eta_X^{\theta}(Q)<\ve$ and $\xi_X^{\theta,\beta}(Q)<\ve$ but $\beta_{X}(B_{Q})\geq  A\beta$ for some large constant $A>0$.

 By Lemma \ref{l:flat-ball}, there is a ball $B'\subseteq c_{0} B_{Q}$ centered on $X$ so that $\beta_{X}(B')<\theta'$ and $r_{B'}\geq c_{\theta} \ell(Q)$ for some $c_{\theta}>0$. Let $x\in B'\cap X$ be the center of this ball, so by Theorem \ref{t:Christ} $x\in Q$.

Since $\eta_X(Q)<\ve$ and $x\in Q$ there are lines $L_{1}(x),...,L_{d}(x)$ passing through $x$ so that 
\[
\sup_{y\in B(x,2\ell(Q)) \cap (L_{1}(x)\cup\cdots \cup L_{d}(x))} \frac{\dist(y,X)}{\ell(Q)}
< \ve.
\]
Let $V$ be the $d$-plane containing the lines $L_{i}(x)$. Let $U\in \cV$ be so that $\angle(V,U)<\ve$. If $L$ is the line infimizing $\eta_{X}^{U,\theta}(x,2\ell(Q))$ and we pick $\ve<\frac{\theta}{2}$, then
\[
\angle(L,V)\geq \angle(L,U)-\angle(V,U)\geq \theta-\ve >\frac{\theta}{2}
\]
and so 
\begin{equation}
\label{e:etaVV'}
\eta_{X}^{V,\theta/2}(x,2\ell(Q))
=\sup_{y\in L\cap B(x,2\ell(Q))} \frac{\dist(y,X)}{2\ell(Q)}
=\eta_{X}^{U,\theta}(x,2\ell(Q)).
\end{equation}
Since $\beta_{X}(B_{Q})\geq  A\beta$, 
\[
\beta_{X}(x,2\ell(Q))
\stackrel{\eqref{e:betamonotone}}{\gec} \beta_{X}(B_{Q})\geq A\beta.
\]
Thus, for $A$ large enough,
\[
\beta_{X}(x,2\ell(Q),U)\geq  \beta_{X}(x,2\ell(Q))\geq \beta\] 
so by the definitions of $\xi_{X}^{\theta,\beta}$ and $\xi_{X,U}^{\theta,\beta}$

\begin{align*}
\eta_{X}^{V,\theta/2}(x,2\ell(Q))
\stackrel{\eqref{e:etaVV'}}{\leq} \eta_{X}^{U,\theta}(x,2\ell(Q))
=\xi_{X,U}^{\theta,\beta}(x,2\ell(Q))
\leq \xi_{X}^{\theta,\beta}(Q)<\epsilon.
\end{align*}

Thus, there is a line $L_{d+1}(x)$ passing through $x$ so that 
\[
\angle(L_{d+1}(x),V)\geq \frac{\theta}{2}
\]
and 
\[
\sup_{y\in B(x,2\ell(Q)) \cap L_{d+1}(x)} \frac{\dist(y,X)}{\ell(Q)}
< \ve.
\]

But for $\ve$ small enough (depending on $\theta$, and $c_{\theta}$), this implies that $\beta_{X}(B')\gec \theta$ (otherwise, if $\beta_{X}(B')\ll\theta$, then the $d$-plane that best approximates $B'\cap X$ would have to be close to $V$, but then it would not contain $L_{d+1}(x)\cap B'$). Since $\beta_{X}(B')<\theta'$, this is impossible for $\theta'\ll \theta$, which gives a contradiction.
\end{proof}

\begin{proof}[Proof of Lemma \ref{l:wgl}]
By Lemmas \ref{l:d-lines}, Lemma \ref{l:xi-carleson}, and \ref{l:eta-xi-beta}, for $\beta>0$ small enough, there are $\ve>0$ and $\theta>0$ so that 
\begin{align*}
\sum_{Q\subseteq R \atop \beta_{X}(B_{Q})\geq A\beta }|Q|
& \leq \sum_{Q\subseteq R \atop \eta_{X}^{\theta}(Q)\geq \ve}|Q|
+\sum_{Q\subseteq R \atop \beta_{X}(B_{Q})\geq A\beta, \; \eta_{X}^{\theta}(Q)<\ve}|Q|\\
& \stackrel{\eqref{e:manysegments}}{\lec} |R| + \sum_{Q\subseteq R \atop \xi_{X}^{\theta,\beta}(Q)\geq \ve}|Q|\\
&  \leq  |R| + \sum_{V\in \cV} \sum_{Q\subseteq R \atop \xi_{X,V}^{\theta,\beta}(Q)\geq \ve}|Q|
\stackrel{\eqref{e:xi-carleson}}{ \lec} |R|
\end{align*}

\end{proof}

\def\cprime{$'$}


\begin{thebibliography}{Azz16b}


\bibitem[Azz16a]{Azz16a}
J.~Azzam.
\newblock Bi-{L}ipschitz parts of quasisymmetric mappings.
\newblock {\em Rev. Mat. Iberoam.}, 32(2):589--648, 2016.

\bibitem[Azz16b]{Azz16b}
J.~Azzam.
\newblock Sets of absolute continuity for harmonic measure in {NTA} domains.
\newblock {\em Potential Anal.}, 45(3):403--433, 2016.

\bibitem[AM16]{AM16}
J.~Azzam and M.~Mourgoglou.
\newblock A characterization of 1-rectifiable doubling measures with connected
  supports.
\newblock {\em Anal. PDE}, 9(1):99--109, 2016.


\bibitem[Bat15]{Bat15}
D.~Bate.
\newblock Structure of measures in {L}ipschitz differentiability spaces.
\newblock {\em J. Amer. Math. Soc.}, 28(2):421--482, 2015.

\bibitem[BKO19]{BKO19}
D.~Bate, I.~Kangasniemi, and T.~Orponen.
\newblock Cheeger's differentiation theorem via the multilinear kakeya
  inequality.
\newblock {\em arXiv preprint arXiv:1904.00808}, 2019.



\bibitem[BL17]{BL17}
D.~Bate and S~Li.
\newblock Characterizations of rectifiable metric measure spaces.
\newblock {\em Ann. Sci. \'{E}c. Norm. Sup\'{e}r. (4)}, 50(1):1--37, 2017.


\bibitem[Che99]{Che99}
J.~Cheeger.
\newblock Differentiability of lipschitz functions on metric measure spaces.
\newblock {\em Geometric and functional analysis}, 9(3):428--517, 1999.


\bibitem[Chr90]{Chr90}
M.~Christ.
\newblock A {$T(b)$} theorem with remarks on analytic capacity and the {C}auchy
  integral.
\newblock {\em Colloq. Math.}, 60/61(2):601--628, 1990.

\bibitem[Dav88]{Dav88}
G.~David.
\newblock Morceaux de graphes lipschitziens et int\'egrales singuli\`eres sur
  une surface.
\newblock {\em Rev. Mat. Iberoamericana}, 4(1):73--114, 1988.


\bibitem[Dav91]{Dav91}
G.~David.
\newblock {\em Wavelets and singular integrals on curves and surfaces}, volume
  1465 of {\em Lecture Notes in Mathematics}.
\newblock Springer-Verlag, Berlin, 1991.



\bibitem[DS91]{DS}
G.~David and S.~W. Semmes.
\newblock Singular integrals and rectifiable sets in {${\bf R}^n$}: {B}eyond
  {L}ipschitz graphs.
\newblock {\em Ast\'erisque}, (193):152, 1991.

\bibitem[DS93]{of-and-on}
G.~David and S.~W. Semmes.
\newblock {\em Analysis of and on uniformly rectifiable sets}, volume~38 of
  {\em Mathematical Surveys and Monographs}.
\newblock American Mathematical Society, Providence, RI, 1993.

%

\bibitem[Dav16]{Dav16}
G.~C. David.
\newblock Bi-{L}ipschitz pieces between manifolds.
\newblock {\em Rev. Mat. Iberoam.}, 32(1):175--218, 2016.

\bibitem[DPMR17]{DPMR17}
G.~De~Philippis, A.~Marchese, and F.~Rindler.
\newblock On a conjecture of {C}heeger.
\newblock In {\em Measure theory in non-smooth spaces}, Partial Differ. Equ.
  Meas. Theory, pages 145--155. De Gruyter Open, Warsaw, 2017.


\bibitem[Dor85]{Dor85}
J.~R. Dorronsoro.
\newblock A characterization of potential spaces.
\newblock {\em Proc. Amer. Math. Soc.}, 95(1):21--31, 1985.




\bibitem[Hei01]{Heinonen}
J.~Heinonen.
\newblock {\em Lectures on analysis on metric spaces}.
\newblock Universitext. Springer-Verlag, New York, 2001.

\bibitem[HK98]{HK98}
J.~Heinonen and P.~Koskela.
\newblock Quasiconformal maps in metric spaces with controlled geometry.
\newblock {\em Acta Math.}, 181(1):1--61, 1998.


\bibitem[HM12]{HM12}
Hyt{\"o}nen, T.; Martikainen, H. Non-homogeneous {$Tb$} theorem and random
  dyadic cubes on metric measure spaces. \emph{J. Geom. Anal.} \textbf{22}
  (2012), no.~4, 1071--1107.



\bibitem[Jon88]{Jon88}
P.~W. Jones.
\newblock Lipschitz and bi-{L}ipschitz functions.
\newblock {\em Rev. Mat. Iberoamericana}, 4(1):115--121, 1988.




\bibitem[Kei03]{Kei03}
S~Keith.
\newblock Modulus and the {P}oincar\'{e} inequality on metric measure spaces.
\newblock {\em Math. Z.}, 245(2):255--292, 2003.

\bibitem[Kei04]{Kei04}
S.~Keith.
\newblock A differentiable structure for metric measure spaces.
\newblock {\em Adv. Math.}, 183(2):271--315, 2004.


\bibitem[KZ08]{KZ08}
S.~Keith and X.~Zhong.
\newblock The poincar{\'e} inequality is an open ended condition.
\newblock {\em Ann. of Math.(2)}, 167(2):575--599, 2008.

\bibitem[KM16]{KM16-primer}
B.~Kleiner and J.~M. Mackay.
\newblock Differentiable structures on metric measure spaces: a primer.
\newblock {\em Ann. Sc. Norm. Super. Pisa Cl. Sci. (5)}, 16(1):41--64, 2016.

\bibitem[Laa00]{Laa00}
T.~J. Laakso.
\newblock Ahlfors {$Q$}-regular spaces with arbitrary {$Q>1$} admitting weak
  {P}oincar\'{e} inequality.
\newblock {\em Geom. Funct. Anal.}, 10(1):111--123, 2000.


\bibitem[Mat95]{Mattila}
P.~Mattila.
\newblock {\em Geometry of sets and measures in {E}uclidean spaces}, volume~44
  of {\em Cambridge Studies in Advanced Mathematics}.
\newblock Cambridge University Press, Cambridge, 1995.
\newblock Fractals and rectifiability.

\bibitem[Mer16]{Mer16}
J. Merhej.
\newblock Poincar{\'e}-type inequalities and finding good parameterizations.
\newblock {\em Mathematische Zeitschrift}, pages 1--33, 2016.


\bibitem[Mer17]{Mer17}
J.~Merhej.
\newblock On the geometry of rectifiable sets with {C}arleson and
  {P}oincar\'{e}-type conditions.
\newblock {\em Indiana Univ. Math. J.}, 66(5):1659--1706, 2017.

\bibitem[JKV97]{JKV97}
A.M.~Vargas Rey, N.H. Katz, and P.~W. Jones.
\newblock Checkerboards, lipschitz functions and uniform rectifiability.
\newblock {\em Revista matem{\'a}tica iberoamericana}, 13(1):189--210, 1997.


\bibitem[Sem96]{Sem96}
S.~W. Semmes.
\newblock Finding curves on general spaces through quantitative topology, with
  applications to {S}obolev and {P}oincar\'e inequalities.
\newblock {\em Selecta Math. (N.S.)}, 2(2):155--295, 1996.



\bibitem[Vil17]{Vil17}
M. Villa.
\newblock Tangent points of d-lower content regular sets and $\beta$-numbers.
\newblock {\em arXiv preprint arXiv:1712.02823, to appear in J. London. Math.
  Soc.}, 2017.
  
  
\bibitem[Vuo88]{Vuorinen}
M.~Vuorinen.
\newblock {\em Conformal geometry and quasiregular mappings}, volume 1319 of
  {\em Lecture Notes in Mathematics}.
\newblock Springer-Verlag, Berlin, 1988.



\end{thebibliography}
\end{document}